\documentclass[12pt]{svjour3}
\usepackage{amsfonts}
\usepackage{amsmath}
\usepackage{amssymb}
\usepackage[T1]{fontenc}
\usepackage{graphicx}

\newcommand{\Sc}{\mathcal{S}}

\newcommand{\Rc}{\mathcal{R}}

\newcommand{\Nbb}{\mathbb{N}}
\newcommand{\Rbb}{\mathbb{R}}

\newcommand{\norm}[1][\cdot]{ \Vert #1 \Vert }
\begin{document}

\title{Proper Generalized Decomposition for Nonlinear Convex
Problems in Tensor Banach Spaces \thanks{This work is supported by
the ANR (French National Research Agency, grants ANR-2010-COSI-006
and ANR-2010-BLAN-0904) and by the PRCEU-UCH30/10 grant of the
Universidad CEU Cardenal Herrera.}}

\titlerunning{PGD for Nonlinear Convex Problems}

\author{Antonio Falc\'o \and Anthony Nouy}

\institute{
A. Falc\'o \at
Departamento de Ciencias, F\'{\i}sicas, Matem\'aticas y de la
Computaci\'on, \\ Universidad CEU Cardenal Herrera, \\
San Bartolom\'e 55 \\
46115 Alfara del Patriarca (Valencia), Spain.\\
\email{afalco@uch.ceu.es} \and A. Nouy \at LUNAM Universit\'e, GeM,
UMR CNRS 6183, Ecole Centrale Nantes, Universit\'e de Nantes\\
1 rue de la No\"e, BP 92101,\\
44321 Nantes Cedex 3, France.\\
\email{anthony.nouy@ec-nantes.fr}
}

\maketitle

\begin{abstract}
Tensor-based methods are receiving a growing interest in scientific
computing for the numerical solution of problems defined in high
dimensional tensor product spaces. A family of methods called Proper
Generalized Decompositions methods have been recently introduced for
the a priori construction of tensor approximations of the solution
of such problems. In this paper, we give a mathematical analysis of
a family of progressive and updated Proper Generalized
Decompositions for a particular class of problems associated with
the minimization of a convex functional over a reflexive tensor
Banach space. \subclass{65K10,49M29} 
\end{abstract}

\section{Introduction}
Tensor-based methods are receiving a growing interest in scientific
computing for the numerical solution of problems defined in high
dimensional tensor product spaces, such as partial differential
equations arising from stochastic calculus (e.g. Fokker-Planck
equations) or quantum mechanics (e.g. Schr\"odinger equation),
stochastic parametric partial differential equations in uncertainty
quantification with functional approaches, and many mechanical or
physical models involving extra parameters (for parametric
analyses),$\ldots .$ For such problems, classical approximation methods
based on the a priori selection of approximation bases suffer from
the so called ``curse of dimensionality'' associated with the
exponential (or factorial) increase in the dimension of
approximation spaces. Tensor-based methods consist in approximating
the solution $\mathbf{u}\in V$ of a problem, where $V$ is a tensor
space generated by $d$ vector spaces $V_j$ (assume e.g. $V_j =
\mathbb{R}^{n_j}$)\footnote{More precisely, $V$ is the closure with
respect to a norm $\norm{}$ of the algebraic tensor space
$\mathbf{V} = \left._a \bigotimes_{j=1}^d V_j \right.=\mathrm{span}
\left\{ \bigotimes_{j=1}^{d} v^{(j)}: v^{(j)} \in V_{j} \text{ and }
1 \le j \le d\ \right\}$}, using separated representations of the
form
\begin{align} \mathbf{u}\approx \mathbf{u}_m = \sum_{i=1}^m
w_i^{(1)}\otimes \ldots \otimes w_i^{(d)},\quad w_i^{(j)}\in V_j
\label{eq:decomp_intro}
\end{align}
where $\otimes$ represents the Kronecker product. The functions
$w_i^{(j)}$ are not a priori selected but are chosen in an optimal
way regarding some properties of $\mathbf{u}$.
\par A first family of numerical methods based on
classical constructions of tensor approximations
\cite{KOL09,GRA10,OSE10} have been recently investigated for the
solution of high-dimensional partial differential equations
\cite{HAC08b,BAL10,KHO10,MAT11}. They are based on the systematic
use of tensor approximations inside classical iterative solvers.
Another family of methods, called Proper Generalized Decomposition
(PGD) methods \cite{LAD10,CHI10,NOU10,NOU10b,FAL10b}, have been
introduced for the direct construction of representations of type
\eqref{eq:decomp_intro}. PGD methods introduce alternative
definitions of tensor approximations, not based on natural best
approximation problems, for the approximation to be computable
without a priori information on the solution $\mathbf{u}$. The
particular structure of approximation sets allows the interpretation
of PGDs as generalizations of Proper Orthogonal Decomposition (or
Singular Value Decomposition, or Karhunen-Lo\`eve Decomposition) for
the a priori construction of a separated representation
$\mathbf{u}_m$ of the solution. They can also be interpreted as a
priori model reduction techniques in the sense that they provide a
way for the a priori construction of optimal reduced bases for the
representation of the solution. Several definitions of PGDs have
been proposed. Basic PGDs are based on a progressive construction of
the sequence $\mathbf{u}_m$, where at each step, an additional
elementary tensor $\otimes_{k=1}^d w_{m}^{(k)}$ is added to the
previously computed decomposition $\mathbf{u}_{m-1}$
\cite{LAD99b,AMM07,NOU07}. Progressive definitions of PGDs can thus
be considered as Greedy algorithms \cite{TEM08} for constructing
separated representations \cite{LEB09,AMM10}. A possible improvement
of these progressive decompositions consists in introducing some
updating steps in order to capture an approximation of the optimal
decomposition, which would be obtained by defining the whole set of
functions simultaneously (and not progressively). For many
applications, these updating strategies allow recovering good
convergence properties of separated representations
\cite{NOU08b,NOU10b,NOU10}.
\\
\par In \cite{LEB09}, convergence
results are given for the progressive Proper Generalized
Decomposition in the case of the high-dimensional Laplacian problem.
In \cite{FAL10b}, convergence is proved in the more general setting
of linear elliptic variational problems in tensor Hilbert spaces.
The progressive PGD is interpreted as a generalized singular value
decomposition with respect to the metric induced by the operator,
which is not necessarily a crossnorm on the tensor product space.
\\
\par In this paper, we propose a theoretical
analysis of progressive and updated Proper Generalized
Decompositions for a class of problems associated with the
minimization of an elliptic and differentiable functional $J$,
$$
J(\mathbf{u}) = \min_{\mathbf{v}\in V} J(\mathbf{v}),
$$
where $V$ is a reflexive tensor Banach space. In this context,
progressive PGDs consist in defining a sequence of approximations
$\mathbf{u}_m \in V$ defined by
$$
\mathbf{u}_m = \mathbf{u}_{m-1} + \mathbf{z}_m,\quad
\mathbf{z}_m\in\Sc_1
$$
where $\Sc_1$ is a tensor subset with suitable properties (e.g.
rank-one tensors subset, Tucker tensors subset, ...), and where
$\mathbf{z}_m$ is an optimal correction in $\Sc_1$ of
$\mathbf{u}_{m-1}$, defined by
$$
J(\mathbf{u}_{m-1}+\mathbf{z}_m) = \min_{\mathbf{z}\in \Sc_1}
J(\mathbf{u}_{m-1} + \mathbf{z})$$ Updated progressive PGDs consist
in correcting successive approximations by using  the information
generated in the previous steps. At step $m$, after having computed
an optimal correction $\mathbf{z}_m\in \Sc_1$ of $\mathbf{u}_{m-1}$,
a linear (or affine) subspace $U_m \subset V$ such that
$\mathbf{u}_{m-1}+\mathbf{z}_m \in U_m$ is generated from the
previously computed information, and the next approximation
$\mathbf{u}_m$ is defined by
$$
J(\mathbf{u}_m) = \min_{\mathbf{v}\in U_m} J(\mathbf{v}) \le
J(\mathbf{u}_{m-1} + \mathbf{z}_m)
$$

 The outline of
the paper is as follows. In section \ref{sec:tensor_product}, we
briefly recall some classical properties of tensor Banach spaces. In
particular, we introduce some assumptions on the weak topology of
the tensor Banach space in order for the (updated) progressive PGDs
to be well defined (properties of subsets $\Sc_1$). In section
\ref{sec:problem}, we introduce a class of convex minimization
problems on Banach spaces in an abstract setting. In section
\ref{sec:progressive_pgd}, we introduce and analyze the progressive
PGD (with or without updates) and we provide some general
convergence results.
While working on this paper, the authors became aware of the work
\cite{CAN10}, which provides a convergence proof for the purely
progressive PGD when working on tensor Hilbert spaces. The present
paper can be seen as an extension of the results of \cite{CAN10} to
the more general framework of tensor Banach spaces and to a larger
family of PGDs, including updating strategies and a general
selection of tensor subsets $\Sc_1$.
In section \ref{sec:examples}, we present some classical examples of
applications of the present results: best approximation in $L^p$
tensor spaces (generalizing the multidimensional singular value
decomposition to $L^p$ spaces), solution of $p$-Laplacian problem,
and solution of elliptic variational problems (involving
inequalities or equalities).

\section{Tensor Banach spaces}
\label{sec:tensor_product}

We first consider the definition of the algebraic tensor space
$_{a}\bigotimes _{j=1}^{d}V_{j}$ generated from Banach spaces
$V_{j}$ $\left( 1\leq j\leq d\right)$ equipped with norms
$\Vert\cdot \Vert_j$. As underlying field we choose $\mathbb{R},$
but the results hold also for $\mathbb{C}$. The suffix `$a$' in
$_{a}\bigotimes_{j=1}^{d}V_{j}$ refers to the `algebraic' nature. By
definition, all elements of
\[
\mathbf{V}:=\left.  _{a}\bigotimes_{j=1}^{d}V_{j}\right.
\]
are \emph{finite} linear combinations of elementary tensors $\mathbf{v}%
=\bigotimes_{j=1}^{d}v^{(j)}$ $\left(  v^{(j)}\in V_{j}\right)  .$

A typical
representation format is the Tucker or tensor subspace format%
\begin{equation}
\mathbf{u}=\sum_{\mathbf{i}\in\mathbf{I}}\mathbf{a}_{\mathbf{i}}%
\bigotimes_{j=1}^{d}b_{i_{j}}^{(j)}, \label{(Tucker}%
\end{equation}
where $\mathbf{I}=I_{1}\times\ldots\times I_{d}$ is a multi-index
set with $I_{j}=\{1,\ldots,r_{j}\},$ $r_{j}\leq\dim(V_{j}),$
$b_{i_{j}}^{(j)}\in V_{j}$ $\left(  i_{j}\in I_{j}\right)  $ are
linearly independent (usually orthonormal) vectors, and
$\mathbf{a}_{\mathbf{i}}\in\mathbb{R}$. Here, $i_{j}$ are the
components of $\mathbf{i}=\left(  i_{1},\ldots,i_{d}\right)  $. The
data size is determined by the numbers $r_{j}$ collected in the
tuple $\mathbf{r}:=\left( r_{1},\ldots,r_{d}\right)  $. The set of
all tensors representable by
(\ref{(Tucker}) with fixed $\mathbf{r}$ is%
\begin{equation}
\mathcal{T}_{\mathbf{r}} (\mathbf{V}) :=\left\{  \mathbf{v}\in\mathbf{V}:%
\begin{array}
[c]{l}%
\text{there are subspaces }U_{j}\subset V_{j}\text{ such that}\\
\dim(U_{j})=r_{j}\text{ and }\mathbf{v}\in\mathbf{U}:=\left.  _{a}%
\bigotimes_{j=1}^{d}U_{j}\right.  .
\end{array}
\text{ }\right\}  \label{(Tr}%
\end{equation}
To simplify the notations, the set of rank-one tensors (elementary
tensors) will be denoted by
$$
\mathcal{R}_1(\mathbf{V}) := \mathcal{T}_{(1,\ldots,1)}(\mathbf{V})
= \left\{\otimes_{k=1}^d w^{(k)} : w^{(k)}\in V_k \right\}.
$$
By definition, we then have $\mathbf{V} = \mathrm{span} \,
\mathcal{R}_1(\mathbf{V}).$ We also introduce the set of rank-$m$
tensors defined by
$$\mathcal{R}_m(\mathbf{V}) := \left\{\sum_{i=1}^m \mathbf{z}_i:
 \mathbf{z}_i\in \mathcal{R}_1(\mathbf{V}) \right\}.$$

We say that $\mathbf{V}_{\left\Vert \cdot\right\Vert }$ is a
\emph{Banach tensor space} if there exists an algebraic tensor space
$\mathbf{V}$ and a norm $\left\Vert \cdot\right\Vert $ on
$\mathbf{V}$ such that $\mathbf{V}_{\left\Vert \cdot\right\Vert }$
is the completion of $\mathbf{V}$ with respect to the norm
$\left\Vert \cdot\right\Vert $, i.e.
\[
\mathbf{V}_{\left\Vert \cdot\right\Vert }:=\left.  _{\left\Vert
\cdot \right\Vert }\bigotimes_{j=1}^{d}V_{j}\right.
=\overline{\left. _{a}\bigotimes\nolimits_{j=1}^{d}V_{j}\right.
}^{\left\Vert \cdot\right\Vert }.
\]
If $\mathbf{V}_{\left\Vert \cdot\right\Vert }$ is a Hilbert space,
we say that $\mathbf{V}_{\left\Vert \cdot\right\Vert }$ is a
\emph{Hilbert tensor space}.

\subsection{Topological properties of Tensor Banach spaces}

Observe that $\mathrm{span} \, \mathcal{R}_1(\mathbf{V})$ is dense
in $\mathbf{V}_{\|\cdot \|}.$ Since $\mathcal{R}_1(\mathbf{V})
\subset \mathcal{T}_{\mathbf{r}}(\mathbf{V})$ for all $\mathbf{r}
\ge (1,1,\ldots,1),$ then $\mathrm{span} \,
\mathcal{T}_{\mathbf{r}}(\mathbf{V})$ is also dense in
$\mathbf{V}_{\|\cdot \|}.$

Any norm $\left\Vert \cdot\right\Vert $ on $\left.  _{a}\bigotimes_{j=1}%
^{d}V_{j}\right.  $ satisfying%
\begin{equation}
\left\Vert \bigotimes\nolimits_{j=1}^{d}v^{(j)}\right\Vert =\prod
\nolimits_{j=1}^{d}\Vert v^{(j)}\Vert_{j}\qquad\text{for all
}v^{(j)}\in
V_{j}\text{ }\left(  1\leq j\leq d\right)  \label{(rcn a}%
\end{equation}
is called a \emph{crossnorm}.

\begin{remark}
\label{tensor product continuity}Eq. (\ref{(rcn a}) implies the
inequality
$\Vert\bigotimes\nolimits_{j=1}^{d}v^{(j)}\Vert\lesssim\prod\nolimits_{j=1}%
^{d}\Vert v^{(j)}\Vert_{j}$ which is equivalent to the continuity of
the
tensor product mapping%
\begin{equation}
\bigotimes
:\mathop{\mathchoice{\raise-0.22em\hbox{\huge $\times$}} {\raise-0.05em\hbox{\Large $\times$}}{\hbox{\large $\times$}}{\times}}_{j=1}%
^{d}\left(  V_{j},\left\Vert \cdot\right\Vert _{j}\right)
\longrightarrow \left(  \left.  _{a}\bigotimes_{j=1}^{d}V_{j}\right.
,\left\Vert
\cdot\right\Vert \right)  , \label{(Tensorproduktabbildung}%
\end{equation}
given by $\otimes\left(  (v^{(1)},\ldots,v^{(d)})\right)  =
\otimes_{j=1}^{d}%
v^{(j)},$ where $(X,\|\cdot\|)$ denotes a vector space $X$ equipped
with norm $\|\cdot\|$.
\end{remark}

As usual, the dual norm to $\left\Vert \cdot\right\Vert $ is denoted
by $\left\Vert \cdot\right\Vert ^{\ast}$. If $\left\Vert
\cdot\right\Vert $ is a crossnorm and also $\left\Vert
\cdot\right\Vert ^{\ast}$ is a crossnorm on $\left.  _{a}\bigotimes_{j=1}%
^{d}V_{j}^{\ast}\right.  $, i.e.%
\begin{equation}
\left\Vert \bigotimes\nolimits_{j=1}^{d}\varphi^{(j)}\right\Vert ^{\ast}%
=\prod\nolimits_{j=1}^{d}\Vert\varphi^{(j)}\Vert_{j}^{\ast}\qquad\text{for
all }\varphi^{(j)}\in V_{j}^{\ast}\text{ }\left(  1\leq j\leq
d\right)  ,
\label{(rcn b}%
\end{equation}
$\left\Vert \cdot\right\Vert $ is called a \emph{reasonable
crossnorm}.
Now, we introduce the following norm.

\begin{definition}
Let $V_{j}$ be Banach spaces with norms $\left\Vert \cdot\right\Vert
_{j}$
for $1\leq j\leq d.$ Then for $\mathbf{v}\in\mathbf{V}=\left.  _{a}%
\bigotimes_{j=1}^{d} V_{j}\right.  $, we define the norm $\left\Vert
\cdot\right\Vert
_{\vee}$ by%
\begin{equation}
\left\Vert \mathbf{v}\right\Vert _{\vee}:=\sup\left\{
\frac{\left\vert \left(
\varphi^{(1)}\otimes\varphi^{(2)}\otimes\ldots\otimes\varphi^{(d)}\right)
(\mathbf{v})\right\vert }{\prod_{j=1}^{d}\Vert\varphi^{(j)}\Vert_{j}^{\ast}%
}:0\neq\varphi^{(j)}\in V_{j}^{\ast},1\leq j\leq d\right\}  .
\label{(Norm ind*(V1,...,Vd)}%
\end{equation}
\end{definition}

The following proposition has been proved in \cite{FalcoHackbusch}.

\begin{proposition}
\label{tr weakly closed}Let $\mathbf{V}_{\left\Vert \cdot\right\Vert
}$ be a Banach tensor space with a norm satisfying $\|\cdot\|
\gtrsim \left\Vert \cdot \right\Vert_{\vee}$ on $\mathbf{V}.$ Then
the set $\mathcal{T}_{\mathbf{r}}(\mathbf{V})$ is weakly closed.
\end{proposition}

\subsection{Examples}\label{several_examples}

\subsubsection{The Bochner spaces}\label{BochnerS}
Our first example, the Bochner spaces, are a generalization of the
concept of $L^p$-spaces to functions whose values lie in a Banach
space which is not necessarily the space $\mathbb{R}$ or
$\mathbb{C}.$

Let $X$ be a Banach space endowed with a norm $\norm_X$. Let
$I\subset \Rbb^{s}$ and $\mu$ a finite measure on $I$ (e.g. a
probability measure). Let us consider the Bochner space
$L^p_\mu(I;X)$, with $1\le p <\infty$, defined by
$$
L^p_\mu(I;X) = \left\{v:I\rightarrow X: \int_{I}  \Vert v(x)
\Vert_X^p d\mu(x) <\infty \right\},
$$
and endowed with the norm
$$
\norm[v]_{\Delta_p} = \left(\int_{I}  \Vert v(x) \Vert_X^p
d\mu(x)\right)^{1/p}
$$
We now introduce the tensor product space
$\mathbf{V}_{\norm[\cdot]_{\Delta_p}} = X
\otimes_{\norm[\cdot]_{\Delta_p}} L^p_\mu(I).$ For $1 \le p <
\infty,$ the space $L^p_\mu(I;X)$ can be identified with
$\mathbf{V}_{\norm[\cdot]_{\Delta_p}}$ (see Section 7, Chapter 1 in
\cite{DefantFloret}). Moreover, the following proposition can be
proved (see Proposition~7.1 in \cite{DefantFloret}):

\begin{proposition}\label{DF}
For $1 \le p < \infty$, the norm $\norm[\cdot]_{\Delta_p}$ satisfies
$\|\cdot\|_{\Delta_p} \gtrsim \left\Vert \cdot \right\Vert _{\vee}$
on $X \otimes_a L^p_\mu(I).$
\end{proposition}
By Propositions \ref{DF} and \ref{tr weakly closed}, we then
conclude:

\begin{corollary}\label{prop:S1Bochner_closed}
For $1 \le p < \infty$, the set $\mathcal{T}_{\mathbf{r}}\left( X
\otimes_{a} L^p_\mu(I) \right)$ is weakly closed in $L^p_\mu(I;X).$
In particular, for $K \subset \mathbb{R}^{k},$ we have that
$\mathcal{T}_{\mathbf{r}}\left( L^p_{\nu}\left(K\right) \otimes_{a}
L^p_{\mu}(I) \right)$ and $\Rc_1\left( L^p_{\nu}\left(K\right)
\otimes_{a} L^p_{\mu}(I) \right)$ are weakly closed sets in
$L^p_{\nu \otimes \mu}\left(K \times I \right).$
\end{corollary}

\subsubsection{The Sobolev spaces}\label{sec:sobolev}

Let $\Omega = \Omega_1\times \hdots \times \Omega_d\subset \Rbb^d$,
with $\Omega_k\subset\Rbb$. Let $\alpha\in \Nbb^d$ denote a
multi-index and $\vert \alpha\vert=\sum_{k=1}^d \alpha_k$.
$D^\alpha(\mathbf{u}) =
\partial_{x_1}^{\alpha_1} \hdots \partial_{x_d}^{\alpha_d}(\mathbf{u})$
denotes a partial derivative of $\mathbf{u}(x_1,\hdots,x_d)$ of
order $\vert \alpha \vert$. For a fixed $1 \le p < \infty,$ we
introduce the Sobolev space
$$H^{m,p}(\Omega) = \{\mathbf{u}\in L^p(\Omega):D^\alpha(\mathbf{u})
\in L^p(\Omega),0\le
\vert\alpha\vert \le m\}$$ equipped with the norm
$$
\norm[\mathbf{v}]_{m,p} = \sum_{0\le  \vert\alpha\vert \le m}
\norm[D^\alpha(\mathbf{v})]_{L^p(\Omega)}
$$
We let $V_k =H^{m,p}(\Omega_k)$, endowed with norms $\norm_{m,p;k}$
defined by
$$
\norm[w]_{m,p;k} = \sum_{j=0}^m \norm[\partial_{x_k}^j
(w)]_{L^p(\Omega_k)}.
$$
Then we have the following equality
\[
H^{m,p}(\Omega)=\left.  _{\left\Vert \cdot\right\Vert _{m,p}}%
\bigotimes_{j=1}^{d}H^{m,p}(\Omega_{j})\right..
\]

A first result is the following.

\begin{proposition}\label{prop:S1W1p_closed}
For $1<p<\infty,$ $m \ge 0$ and $\Omega =\Omega_1\times \hdots \times
\Omega_d$, the set
$$
\Rc_1\left( \left._a \bigotimes_{j=1}^{d}H^{m,p}(\Omega_{j})\right.
\right) = \left\{ \otimes_{k=1}^d w^{(k)}:w^{(k)}\in
H^{m,p}(\Omega_k)\right\},
$$
is weakly closed in $(H^{m,p}(\Omega),\norm[\cdot]_{m,p})$.
\end{proposition}

To prove the above proposition we need the following two lemmas.

\begin{lemma}\label{lem:Wmp0}
Assume $1 < p < \infty$ and $\Omega =\Omega_1\times \hdots \times
\Omega_d.$ Then the set
$\Rc_1\left( \left._a \bigotimes_{j=1}^{d}L^{p}(\Omega_{j})\right.
\right)$ is weakly closed in $L^p(\Omega).$
\end{lemma}

\begin{proof}
Let $\{\mathbf{v}_n\}_{n \in \mathbb{N}},$ with $\mathbf{v}_n =
\otimes_{j=1}^d v_n^{(j)},$ be a sequence in $\Rc_1\left( \left._a
\bigotimes_{j=1}^{d}L^{p}(\Omega_{j})\right. \right)$ that weakly
converges to an element $\mathbf{v} \in L^p(\Omega).$ Then the
sequence $\{\mathbf{v}_n\}_{n \in \mathbb{N}}$ is bounded in
$L^p(\Omega),$ and also the sequences $\{v_n^{(j)}\}_{n  \in
\mathbb{N}}\subset L^p(\Omega_j)$ for each $j \in \{1,2,\ldots,d\}.$
Then, for each $j \in \{1,2,\ldots,d\},$ we can extract a
subsequence, namely $\{v_{n_k}^{(j)}\}_{k \in \mathbb{N}},$ that
weakly converges to some $v^{(j)} \in L^p(\Omega_j).$ Since weak
convergence in $L^p(\Omega_j)$ implies the convergence in
distributional sense, that is, the subsequence $\{v_{n_k}^{(j)}\}_{k
\in \mathbb{N}}$ converges to $v^{(j)}$ in $\mathcal{D}'(\Omega_j).$
From Proposition~6.2.3 of \cite{Blanchard2003}, we have that
$\{\otimes_{j=1}^d v_{n_k}^{(j)}\}_{k \in \mathbb{N}}$ converges to
$\otimes_{j=1}^d v^{(j)}$ in $\mathcal{D}'(\Omega).$ By uniqueness
of the limit, we obtain the desired result. \qed \end{proof}

\begin{lemma}\label{lem:Wmp}
Assume $1<p<\infty,$ $m \ge 1$ and $\Omega =\Omega_1\times \hdots \times
\Omega_d.$ For any measurable functions $w_k:\Omega_k\rightarrow \Rbb$ such
that $\otimes_{k=1}^d w_k \neq \mathbf{0},$ we have $\otimes_{k=1}^d
w_k \in H^{m,p}(\Omega)$ if and only if $w_k\in H^{m,p}(\Omega_k)$
for all $k\in\{1\hdots d\}$.
\end{lemma}
\begin{proof}
Suppose that $w_k\in H^{m,p}(\Omega_k)$ for all $k\in\{1\hdots d\}$.
Since
\begin{align*}
\norm[\otimes_{k=1}^d w_k]_{m,p} &=  \sum_{0\le \vert \alpha \vert \le m}
\prod_{k=1}^d \Vert \partial_{x_k}^{\alpha_k}(w_k)
\Vert_{L^p(\Omega_k)}
\\
&\le\sum_{\alpha\in\{0,\hdots,m\}^d } \prod_{k=1}^d \Vert
\partial_{x_k}^{\alpha_k}(w_k) \Vert_{L^p(\Omega_k)}\\
&= \prod_{k=1}^d \left( \sum_{j=0}^m \Vert
\partial_{x_k}^{j}(w_k) \Vert_{L^p(\Omega_k)}\right)= \prod_{k=1}^d \norm[
w_k]_{m,p;k},
\end{align*}
we have
$\otimes_{k=1}^d w_k \in H^{m,p}(\Omega)$.

Conversely, if
$\otimes_{k=1}^d w_k \in H^{m,p}(\Omega)$, then
\begin{align*}
\norm[\otimes_{k=1}^d w_k]_{m,p}  =& \sum_{0\le  \vert\alpha\vert \le m}
\norm[D^\alpha(\otimes_{k=1}^d w_k)]_{L^p(\Omega)} <\infty
\end{align*}
which implies that $\norm[D^\alpha(\otimes_{k=1}^d
w_k)]_{L^p(\Omega)}<\infty $ for all $\alpha$ such that $0\le
\vert\alpha\vert \le m$. Taking $\alpha=(0,\hdots,0)$, we obtain
$$
\norm[\otimes_{k=1}^d w_k]_{L^p(\Omega)} = \prod_{k=1}^d
\norm[w_k]_{L^p(\Omega_k)} <\infty
$$
and therefore $\norm[w_k]_{L^p(\Omega_k)}<\infty$ for all $k$. Now,
for $k\in\{1\hdots d\}$, taking $\alpha = (\hdots,0,j,0,\hdots)$
such that $\alpha_k=j$, with $1\le j\le m$, and $\alpha_l=0$ for
$l\neq k$, we obtain
$$
\norm[D^{\alpha}(\otimes_{l=1}^d w_l)]_{L^p(\Omega)} = \Vert
\partial^j_{x_k} w_k\Vert_{L^p(\Omega_k)} \prod_{l\neq k}
\norm[w_l]_{L^p(\Omega_l)}
$$
and then $\norm[\partial_{x_k}^j w_k]_{p,\Omega_k}<\infty$ for all
$j\in\{1,\hdots,m\}$. Therefore $w_k\in H^{m,p}(\Omega_k)$ for all
$k\in\{1\hdots d\}$.
\qed \end{proof}

\noindent \emph{Proof of Proposition~\ref{prop:S1W1p_closed} }
For $m=0$ the proposition follows from Lemma~\ref{lem:Wmp0}. Now,
assume $m \ge 1,$ and let us consider a sequence
$$
\{\mathbf{z}_n\}_{n \in \mathbb{N}} \subset \Rc_1\left( \left._a
\bigotimes_{j=1}^{d}H^{m,p}(\Omega_{j})\right. \right)
$$
that weakly converges
to an element $\mathbf{z}\in H^{m,p}(\Omega)$. Since
$$
\Rc_1\left( \left._a \bigotimes_{j=1}^{d}H^{m,p}(\Omega_{j})\right.
\right) \subset \Rc_1\left( \left._a
\bigotimes_{j=1}^{d}L^{p}(\Omega_{j})\right. \right),$$ we have
$\mathbf{z} \in \Rc_1\left( \left._a
\bigotimes_{j=1}^{d}L^{p}(\Omega_{j})\right. \right)$ because, from
Lemma~\ref{lem:Wmp0}, the latter set is weakly closed in
$L^p(\Omega)$. Therefore, there exist $w_k \in L^p(\Omega_k)$ such
that $\mathbf{z} = \otimes_{k=1}^d w_k$. Since $\mathbf{z} \in
H^{m,p}(\Omega)$, from Lemma \ref{lem:Wmp}, $w_k\in
H^{m,p}(\Omega_k)$ for $1 \le k \le d,$ and therefore $\mathbf{z} =
\otimes_{k=1}^d w_k \in \Rc_1\left( \left._a
\bigotimes_{j=1}^{d}H^{m,p}(\Omega_{j})\right. \right)$. \qed

From \cite{FalcoHackbusch} it follows the following statement.

\begin{proposition}
The set $\mathcal{T}_{\mathbf{r}}\left(\left._{a}%
\bigotimes_{j=1}^{d}H^{m,2}(\Omega_{j})\right.\right)$ is weakly
closed in $H^{m,2}(\Omega).$
\end{proposition}

\section{Optimization of functionals over Banach spaces}\label{sec:problem}

Let $V$ be a reflexive Banach space, endowed with a norm $\Vert\cdot
\Vert$. We denote by $V^\ast$ the dual space of $V$ and we denote by
$\langle \cdot,\cdot \rangle:V^\ast\times V\rightarrow \Rbb$ the
duality pairing. We consider the optimization problem
\begin{align}
J(u) = \min_{v\in V} J(v)\label{eq:prob}\tag{$\pi$}
\end{align}
where $J:V\rightarrow \Rbb$ is a given functional.

\subsection{Some useful results on minimization of functionals over Banach spaces}

In the sequel, we will introduce approximations of \eqref{eq:prob}
by considering an optimization on subsets $M\subset V$, i.e.
\begin{align}
\inf_{v\in M} J(v)\label{eq:inf_problem}
\end{align}
We here recall classical theorems for the existence of a minimizer
(see e.g. \cite{EKE99}).

We recall that a sequence $v_{m}\in V$ is \emph{weakly convergent}
if $\lim_{m\rightarrow\infty}\langle \varphi,  v_{m}  \rangle$
exists for all $\varphi\in V^{\ast}.$ We say that $\left(
v_{m}\right) _{m\in\mathbb{N}}$ \emph{converges weakly to} $v\in V$
if $\lim_{m\rightarrow\infty} \langle \varphi, v_{m} \rangle
=\langle \varphi, v\rangle$ for all $\varphi\in V^{\ast}$. In this
case, we write $v_{m}\rightharpoonup v$.

\begin{definition}
A subset $M\subset V$ is called \emph{weakly closed} if $v_{m}\in M$
and $v_{m}\rightharpoonup v$ implies $v\in M$.
\end{definition}

Note that `weakly closed' is stronger than `closed', i.e., $M$
weakly closed $\Rightarrow$ $M$ closed.

\begin{definition}
We say that a map $J:V \longrightarrow \mathbb{R}$ is weakly
sequentially lower semicontinuous (respectively,  weakly
sequentially continuous) in $M \subset V$ if for all $v \in M$ and
for all $v_{m}\in M$ such that $v_{m}\rightharpoonup v$, it holds
$J(v) \le \lim \inf_{m\rightarrow \infty} J(v_m)$ (respectively,
$J(v) = \lim_{m \rightarrow \infty} J(v_m)).$
\end{definition}

If $J':V \longrightarrow V^\ast$ exists as Gateaux derivative, we
say that $J'$ is strongly continuous when for any sequence $v_n
\rightharpoonup v$ in $V$ it holds that $J'(v_n) \rightarrow J'(v)$
in $V^\ast.$

Recall that the convergence in norm implies the weak convergence.
Thus, $J$ weakly sequentially (lower semi)continuous in $M$
$\Rightarrow$ $J$ (lower semi)continuous in $M.$ It can be shown
(see Proposition 41.8 and Corollary 41.9 in \cite{Zeidler}) the
following result.

\begin{proposition}\label{Zeid}
Let $V$ be a reflexive Banach space and let $J:V \rightarrow \Rbb$
be a functional, then the following statements hold.
\begin{itemize}
\item[(a)] If $J$ is a convex and lower semicontinuous functional, then $J$ is
weakly sequentially lower semicontinuous.
\item[(b)] If $J':V \longrightarrow V^\ast$ exists on $V$ as Gateaux derivative
and is strongly continuous (or compact), then $J$ is weakly sequentially continuous.
\end{itemize}
\end{proposition}

Finally, we have the following two useful theorems.

\begin{theorem}\label{th:bounded_closed}
Assume $V$ is a reflexive Banach space, and assume $M\subset V$ is
bounded and weakly closed. If $J:M \rightarrow \Rbb \cup\{\infty\}$
is weakly sequentially lower semicontinuous, then problem
\eqref{eq:inf_problem} has a solution.
\end{theorem}
\begin{proof}
Let $\alpha = \inf_{v\in A} J(v)$ and $\{v_n\}\subset A$ be a
minimizing sequence. Since $A$ is bounded, $\{v_n\}_{n\in\Nbb}$ is a
bounded sequence in a reflexive Banach space and therefore, there
exists a subsequence $\{v_{n_k}\}_{k\in\Nbb}$ that converges weakly
to an element $u\in V$. Since $A$ is weakly closed, $u\in A$ and
since $J$ is weakly sequentially lower semicontinuous, $J(u)\le
\liminf_{k\to\infty } J(v_{n_k}) = \alpha$. Therefore, $J(u)=\alpha$
and $u$ is solution of the minimization problem.
\qed \end{proof}
We now remove the assumption that $M$ is bounded by adding a
coercivity condition on $J$.
\begin{theorem}\label{th:closed_coercive}
Assume $V$ is a reflexive Banach space, and $M\subset V$ is weakly
closed. If $J:M\rightarrow \Rbb \cup\{\infty\}$ is weakly
sequentially lower semicontinuous and coercive on $M$, then problem
\eqref{eq:inf_problem} has a solution.
\end{theorem}
\begin{proof}
Pick an element $v_0 \in M$ such that $J(v_0)\neq \infty$ and define
$M_0=\{v\in M : J(v)\le J(v_0)\}$. Since $J$ is coercive, $M_0$ is
bounded. Since $M$ is weakly closed and $J$ is weakly sequentially
lower semicontinuous, $M_0$ is weakly closed. The initial problem is
then equivalent to $\inf_{v\in M_0} J(v)$, which admits a solution
from Theorem \ref{th:bounded_closed}.
\qed \end{proof}

\subsection{Convex optimization in Banach spaces}

From now one, we will assume that the functional $J$ satisfies the
following assumptions.
\begin{itemize}
\item[(A1)] $J$ is Fr\'echet differentiable, with Fr\'echet
differential $J':V\rightarrow V^\ast$.
\item[(A2)] $J$ is elliptic, i.e. there exist $\alpha>0$ and $s>1$
such that for all $v,w\in V$;
\begin{align}
\langle J'(v)-J'(w),v-w\rangle \;\ge\; \alpha \Vert
v-w\Vert^s\label{eq:J_strong_convexity}
\end{align}
\end{itemize}
In the following, $s$ will be called the ellipticity exponent of
$J$.
\begin{lemma}\label{lemma_assumption}
Under assumptions (A1)-(A2), we have
\begin{itemize}
\item[(a)] For all $v,w\in V,$
\begin{align}
J(v)-J(w) \ge \langle J'(w),v-w\rangle  +  \frac{\alpha}{s} \Vert
v-w\Vert^s. \label{eq:J_strong_convexity_consequence}
\end{align}
\item[(b)] $J$ is strictly convex.
\item[(c)] $J$ is bounded from below and coercive, i.e. $\lim_{\Vert
v\Vert\rightarrow\infty} J(v)=+\infty.$
\end{itemize}
\end{lemma}
\begin{proof}
\begin{itemize}
\item[(a)] For all $v,w\in V$,
\begin{align*} J(v)-J(w) &= \int_{0}^1 \frac{d}{dt}J(w+t(v-w))
dt = \int_{0}^1
\langle J'(w+t(v-w)),v-w\rangle  dt\\
&=\langle J'(w),v-w\rangle  +\int_{0}^1
\langle J'(w+t(v-w))-J'(w),v-w\rangle  dt\\
&\ge \langle J'(w),v-w\rangle  +\int_{0}^1
\frac{\alpha}{t} \Vert t(v-w)\Vert^s dt\\
&=\langle J'(w),v-w\rangle  + \frac{\alpha}{s} \Vert v-w\Vert^s
\end{align*}
\item[(b)] From (a), we have for $v\neq w$,
$$
J(v)-J(w) \; > \; \langle J'(w),v-w \rangle
$$
\item[(c)] Still from (a), we have for all $v\in V$,
$$
J(v) \ge J(0) +  \langle J'(0),v\rangle + \frac{\alpha}{s} \Vert
v\Vert^s \ge J(0) -  \Vert J'(0) \Vert  \Vert v \Vert +
\frac{\alpha}{s} \Vert v \Vert^s
$$
which gives the coercivity and the fact that $J$ is bounded from
below.
\end{itemize}
\qed \end{proof}

The above properties yield the following classical result.

\begin{theorem}\label{th:solution}
Under assumptions (A1)-(A2), the problem \eqref{eq:prob} admits a
unique solution $u\in V$ which is equivalently characterized by
\begin{align}
\langle J'(u),v\rangle  = 0\quad \forall v\in V \label{eq:euler}
\end{align}
\end{theorem}
\begin{proof} We here only give a sketch of proof of this very classical
result. $J$ is continuous and a fortiori lower semicontinuous. Since
$J$ is convex and lower semicontinuous, it is also weakly
sequentially lower semicontinuous (Proposition \ref{Zeid}(a)). The
existence of a solution then follows from Theorem
\ref{th:closed_coercive}. The uniqueness is given by the strict
convexity of $J$, and the equivalence between \eqref{eq:prob} and
\eqref{eq:euler} classically follows from the differentiability of
$J$.
\qed \end{proof}

\begin{lemma}\label{lem:strong_convergence}
Assume that $J$ satisfies (A1)-(A2). If $\{v_m\}\subset V$ is a
sequence such that $J(v_m)\underset{m\to\infty}{\longrightarrow}
J(u)$, where $u$ is the solution of \eqref{eq:prob}, then
$v_m\rightarrow u$, i.e.
$$
\Vert u -v_m\Vert \underset{m\rightarrow\infty}{\longrightarrow} 0
$$
\end{lemma}
\begin{proof}
By the ellipticity property
\eqref{eq:J_strong_convexity_consequence} of $J$, we have
\begin{align}
J(v_m)-J(u)  \ge \langle J'(u),v_m-u\rangle + \frac{\alpha}{s}\Vert
u-v_{m}\Vert^s = \frac{\alpha}{s}\Vert u-v_{m}\Vert^s.
\end{align}
Therefore,
$$
\frac{\alpha}{s} \Vert u-v_{m}\Vert^s
\le J(v_m)-J(u) 
\underset{m\to \infty}{\longrightarrow} 0,
$$
which ends the proof.
\qed \end{proof}

\section{Progressive Proper Generalized Decompositions
in Tensor Banach Spaces}
\label{sec:progressive_pgd}

\subsection{Definition of progressive
Proper Generalized Decompositions}

We now consider the minimization problem \eqref{eq:prob} of
functional $J$ on a reflexive tensor Banach space $V =
\mathbf{V}_{\|\cdot\|}$. Assume that we have a functional
$J:\mathbf{V}_{\|\cdot\|} \longrightarrow \mathbb{R}$ satisfying
(A1)-(A2) and a weakly closed subset $\Sc_1$ in
$\mathbf{V}_{\|\cdot\|}$ such that
\begin{enumerate}
\item[(B1)] $\Sc_1 \subset \mathbf{V}$, with $\mathbf{0} \in \Sc_1,$
\item[(B2)] for each $\mathbf{v} \in \Sc_1$ we have
$\lambda \mathbf{v} \in \Sc_1$ for all $\lambda \in \mathbb{R},$ and
\item[(B3)] $\mathrm{span} \,\Sc_1$ is dense in $\mathbf{V}_{\|\cdot\|}.$
\end{enumerate}

By using the notation introduced in Section~\ref{several_examples}
we give the following examples.

\begin{example}
Consider $\mathbf{V}_{\|\cdot\|} = L^p_\mu(I;X)$ and $\Sc_1 =
\mathcal{T}_{\mathbf{r}}\left( X \otimes_{a} L^p_\mu(I) \right)$.
\end{example}

\begin{example}
Consider $\mathbf{V}_{\|\cdot\|} = H^{m,2}(\Omega)$ and
$\Sc_1 = \mathcal{T}_{\mathbf{r}}\left(\left._{a}%
\bigotimes_{j=1}^{d}H^{m,2}(\Omega_{j})\right.\right)$.

\end{example}

\begin{example}
Consider $\mathbf{V}_{\|\cdot\|} = H^{m,p}(\Omega)$ and
$\Sc_1 = \mathcal{R}_{1}\left(\left._{a}%
\bigotimes_{j=1}^{d}H^{m,p}(\Omega_{j})\right.\right).$
\end{example}

The set $\mathcal{S}_1$ can be used to characterize the solution of
problem \eqref{eq:prob} as shown by the following result.

\begin{lemma}\label{lem:carac_sol_S1}
Assume that $J$ satisfies (A1)-(A2) and
let $\mathbf{u}^* \in \mathbf{V}_{\|\cdot\|}$ satisfying
\begin{equation}
J(\mathbf{u}^*) = \min_{\mathbf{z} \in \mathcal{S}_1}
J(\mathbf{u}^*+\mathbf{z}).
\end{equation}
Then $\mathbf{u}^*$ solves \eqref{eq:prob}.
\end{lemma}

\begin{proof}
For all $\gamma\in\Rbb_+$ and $\mathbf{z}\in\Sc_1$,
$$
J(\mathbf{u}^* + \gamma \mathbf{z}) \ge J(\mathbf{u}^*)
$$
and therefore
$$
\langle J'(\mathbf{u}^*),\mathbf{z}\rangle  = \lim_{\gamma\searrow
0} \frac{1}{\gamma} (J(\mathbf{u}^* + \gamma \mathbf{z}) -
J(\mathbf{u}^*))\ge 0
$$
holds for all $\mathbf{z}\in\Sc_1.$ From (B2), we have
$$
\langle J'(\mathbf{u}^*),\mathbf{z}\rangle  \; = \; 0\quad \forall
\mathbf{z}\in\Sc_1,
$$
From (B3), we then obtain
$$
\langle J'(\mathbf{u}^*),\mathbf{v}\rangle  = 0 \quad \forall
\mathbf{v}\in \mathbf{V}_{\|\cdot\|},
$$
and the lemma follows from Theorem \ref{th:solution}.
\qed \end{proof}

In the following, we denote by $\Sc_m$ the set
$$
\Sc_m = \left\{\sum_{i=1}^m \mathbf{z}_i \;:\; \mathbf{z}_i \in
\Sc_1 \right\}
$$



The next two lemmas will be useful to define a progressive
Proper Generalized Decomposition.

\begin{lemma}\label{lem:closed_equivalent}
For each $\mathbf{v} \in \mathbf{V}_{\|\cdot\|}$, the set
$$
\mathbf{v} + \Sc_1 =\{\mathbf{v}+\mathbf{w} : \mathbf{w} \in \Sc_1
\}
$$
is weakly closed in $\mathbf{V}_{\|\cdot\|}$.
\end{lemma}

\begin{proof}
Assume that $\mathbf{v} + \mathbf{w}_n \rightharpoonup \mathbf{w}$
for some $\{\mathbf{w}_n\}_{n \ge 1} \subset \Sc_1,$ then
$\mathbf{w}_n \rightharpoonup \mathbf{w}-\mathbf{v}$ and since
$\Sc_1$ is weakly closed, $\mathbf{w}-\mathbf{v} \in \Sc_1.$ In
consequence $\mathbf{w} \in \mathbf{v} + \Sc_1$ and the lemma
follows.
\qed \end{proof}

\begin{lemma}[Existence of a $\Sc_1$-minimizer]\label{lem:minimizer}
Assume that $J:\mathbf{V}_{\|\cdot\|} \longrightarrow \mathbb{R}$
satisfies (A1)-(A2). Then for any $\mathbf{v}\in
\mathbf{V}_{\|\cdot\|}$, the following problem admits a solution:
$$\min_{\mathbf{z} \in \Sc_1} J(\mathbf{v} + \mathbf{z})
= \min_{\mathbf{w} \in  \mathbf{v} + \Sc_1} J(\mathbf{w})$$
\end{lemma}
\begin{proof}
Fr\'echet differentiability of $J$ implies that $J$ is continuous
and since $J$ is convex, we have that $J$ is weakly sequentially
lower semicontinuous by Proposition~\ref{Zeid}. Moreover, $J$ is
coercive on $\mathbf{V}_{\|\cdot\|}$ by
Lemma~\ref{lemma_assumption}(c). By
Lemma~\ref{lem:closed_equivalent}, $\mathbf{v} + \Sc_1$ is a weakly
closed subset in
$\mathbf{V}_{\|\cdot\|}.$ Then, 
the existence
of a minimizer follows from Theorem~\ref{th:closed_coercive}.
\qed \end{proof}

\begin{definition}[Progressive PGDs]
\label{def:updated_progressive_pgd}
Assume that $J:\mathbf{V}_{\|\cdot\|} \longrightarrow \mathbb{R}$
satisfies (A1)-(A2), we define a progressive Proper
Generalized Decomposition $\{\mathbf{u}_m\}_{m \ge 1},$ over
$\Sc_1,$  of $\mathbf{u} = \arg \min_{\mathbf{v} \in
\mathbf{V}_{\|\cdot\|}} J(\mathbf{v})$ as follows. We let
$\mathbf{u}_0=\mathbf{0}$ and for $m \ge 1,$ we construct
$\mathbf{u}_m \in \mathbf{V}_{\|\cdot\|}$ from $\mathbf{u}_{m-1} \in
\mathbf{V}_{\|\cdot\|}$ as we show below.
We first find an element $\mathbf{\hat  z}_m \in \Sc_1 \subset
\mathbf{V}$ such that
$$
J(\mathbf{u}_{m-1}+ \mathbf{\hat z}_m) = \min_{\mathbf{z}\in \Sc_1}
J(\mathbf{u}_{m-1} + \mathbf{z}) \quad (*).
$$
Then at each step $m$ and before to update $m$ to $m+1,$
we can choose one of the following strategies denoted by
$c,l$ and $r,$ respectively:
\begin{itemize}
\item[$(c)$] Let $\mathbf{z}_m = \mathbf{\hat  z}_m.$
Define $\mathbf{u}_{m} = \mathbf{u}_{m-1} + \mathbf{ z}_m,$ update
$m$ to $m+1$ and goto $(*).$
\item[$(l)$] Let $\mathbf{z}_m = \mathbf{\hat  z}_m.$
Construct a closed subspace
$\mathbf{U}(\mathbf{u}_{m-1}+\mathbf{z}_m)$ in
$\mathbf{V}_{\|\cdot\|}$ such that $\mathbf{u}_{m-1}+\mathbf{z}_m
\in \mathbf{U}(\mathbf{u}_{m-1}+\mathbf{z}_m)$. Then, define
$$
\mathbf{u}_{m} = \arg\min_{\mathbf{v}\in \mathbf{U}(\mathbf{u}_{m-1}+\mathbf{z}_m)}
J(\mathbf{v}),
$$
update $m$ to $m+1$ and go to $(*).$
\item[$(r)$] Construct a closed
subspace $\mathbf{U}(\mathbf{\hat  z}_m)$ in
$\mathbf{V}_{\|\cdot\|}$ such that $\mathbf{\hat  z}_m \in
\mathbf{U}(\mathbf{\hat  z}_m),$ and define
$$
\mathbf{z}_m = \arg\min_{\mathbf{z}\in \mathbf{U}(\mathbf{\hat
z}_m)} J(\mathbf{u}_{m-1} + \mathbf{z}).
$$
Then, define
$$
\mathbf{u}_{m} = \mathbf{u}_{m-1} +  \mathbf{z}_m=
\arg\min_{\mathbf{v}\in \mathbf{u}_{m-1}+\mathbf{U}(\mathbf{\hat
z}_m)} J(\mathbf{v}),
$$
update $m$ to $m+1$ and go to $(*).$
\end{itemize}
\end{definition}
Strategies of type $(l)$ and $(r)$ are called \textit{updates}.
Observe that to each progressive Proper Generalized Decomposition
$\{\mathbf{u}_m\}_{m \ge 1}$ of $\mathbf{u}$ we can assign a
sequence of symbols (perhaps finite), that we will denote by
$$
\underline{\boldsymbol{\alpha}}(\mathbf{u}) = \boldsymbol{\alpha}_1 \boldsymbol{\alpha}_2
\cdots \boldsymbol{\alpha}_k \cdots
$$
where $\boldsymbol{\alpha}_k \in \{c,l, r \}$ for all
$k=1,2,\ldots.$ That  means that $\mathbf{u}_k$ was obtained without
update if $\boldsymbol{\alpha}_k=c$, or with an update strategy of
type $l$ or $r$ if $\boldsymbol{\alpha}_k=l$ or
$\boldsymbol{\alpha}_k=r$ respectively. In particular, the
progressive PGD defined in \cite{CAN10} coincides with a PGD where
$\boldsymbol{\alpha}_k = c$ for all $k \ge 1.$ Such a decomposition
is called a \textit{purely progressive PGD}, while a decomposition
such that $\boldsymbol{\alpha}_k = l$ or $\boldsymbol{\alpha}_k = r$
for some $k$ is called an \textit{updated progressive PGD}.

\begin{remark}\label{rem:several_updates}
The update $\boldsymbol{\alpha}_m= l$ can be defined with
several updates at each iteration.
Letting $\mathbf{u}_m^{(0)}=\mathbf{u}_{m-1}+\mathbf{\hat  z}_m$, we
introduce a sequence $\{\mathbf{u}_m^{(p)}\}_{p=1}^{d_m}\subset
\mathbf{V}_{\|\cdot\|}$ defined by
$$
\mathbf{u}_m^{(p+1)} =  \arg\min_{\mathbf{v}\in
\mathbf{U}(\mathbf{u}^{(p)}_{m})} J(\mathbf{v})
$$
with $\mathbf{U}(\mathbf{u}^{(p)}_{m})$ being a closed linear
subspace of $\mathbf{V}_{\|\cdot\|}$ which contains
$\mathbf{u}^{(p)}_{m}$. We finally let $\mathbf{u}_m
=\mathbf{u}_m^{(d_m)}$.
\end{remark}

In \cite{FalcoHackbusch} it was introduced the following definition.
For a given $\mathbf{v}$ in the algebraic tensor space $\mathbf{V}$,
the minimal subspaces $U_{j,\min}(\mathbf{v}) \subset V_j$ are given
by the intersection of all subspaces $U_{j} \subset V_j$ satisfying
$\mathbf{v} \in\left._{a}\bigotimes_{j=1}^{d}U_{j}\right..$ It can
be shown \cite{FalcoHackbusch} that
$\left._{a}\bigotimes_{j=1}^{d}U_{j,\min}(\mathbf{v})\right.$ is a
finite dimensional subspace of $\mathbf{V}.$

\begin{example}[Illustrations of updates]
For a given $\mathbf{v}_m\in \mathbf{V}_{\|\cdot\|}$ (e.g.
$\mathbf{v}_m = \mathbf{u}_{m-1} + \mathbf{z}_m$ if
$\boldsymbol{\alpha}_m=l$ or $\mathbf{v}_m = \mathbf{\hat z}_m$  if
$\boldsymbol{\alpha}_m=r$) there are several possible choices for
defining a linear subspace $\mathbf{U}(\mathbf{v}_m)$. Among others,
we have
\begin{itemize}
\item $\mathbf{U}(\mathbf{v}_m) = \left._a
\bigotimes_{j=1}^d U_{j,\min}(\mathbf{v}_m)\right.$. In the case of
$\boldsymbol{\alpha}_m=l,$ all subspaces
$\mathbf{U}(\mathbf{u}_{m-1} + \mathbf{z_m})$ are finite dimensional
and we have that $\mathbf{u}_m \in \mathbf{V}$ for all $m \ge 1.$
\item Assume that $\mathbf{v}_{m} = \sum_{i=1}^{m} \alpha_i \mathbf{z}_i$ for some
$\{\mathbf{z}_1,\ldots,\mathbf{z}_m\} \subset
\mathbf{V}_{\|\cdot\|},$ $\alpha_i\in\Rbb$, $1\le i\le m$. Then we
can define
$$\mathbf{U}(\mathbf{v}_m) =
\text{span }\{\mathbf{z}_1,\ldots,\mathbf{z}_m\}.$$ In the context
of Greedy algorithms for computing best approximations, an update of
type $\boldsymbol{\alpha}_m = r$ by using an orthonormal basis of
$\mathbf{U}(\mathbf{v}_m)$ corresponds to an orthogonal Greedy
algorithm.
\item Assume $\mathbf{v}_m\in\mathbf{V}$. Fix $k \in \{1,2,\ldots,d\}.$
By using $\left._a \bigotimes_{j=1}^d V_j \right. \cong V_k
\otimes_a \left(_a\bigotimes_{j \neq k} V_j\right),$ we can write
$\mathbf{v}_{m} = \sum_{i=1}^m w^{(k)}_i \otimes \left(\bigotimes_{j
\neq k}w_i^{(j)}\right)$ for some elementary tensors $w^{(k)}_i
\otimes \left(\bigotimes_{j \neq k}w_i^{(j)}\right)$ for
$i=1,\ldots,m.$ Then we can define the linear subspace
$$
\mathbf{U}(\mathbf{v}_m) = \left\{\sum_{i=1}^m v^{(k)}_i \otimes
\left(\bigotimes_{j \neq k}w_i^{(j)}\right) : v^{(k)}_i \in V_k, \,
1 \le i \le m \right\}.
$$
The minimization on $\mathbf{U}(\mathbf{v}_m)$ corresponds to an
update of functions along dimension $k$ (functions in the Banach
space $V_k$). Following the remark \ref{rem:several_updates},
several updates could be defined by choosing a sequence of updated
dimensions.
\end{itemize}
\end{example}

\subsection{On the convergence of the progressive PGDs}

Now, we study the convergence of progressive PGDs.
Recall that $\mathbf{\hat z}_m \in\Sc_1$ is a solution of
$$
J(\mathbf{u}_{m-1}+\mathbf{\hat z}_m) = \min_{\mathbf{z}\in\Sc_1}
J(\mathbf{u}_{m-1}+\mathbf{z}),
$$
For  $\boldsymbol{\alpha}_m = c,$ we have $\mathbf{z}_m = \mathbf{\hat z}_m$
and $\mathbf{u}_m=\mathbf{u}_{m-1}+\mathbf{z}_m,$ so that
$$J(\mathbf{u}_m) = J(\mathbf{u}_{m-1}+\mathbf{z}_m) = J(\mathbf{u}_{m-1}+\mathbf{\hat
z}_m)
$$
For $\boldsymbol{\alpha}_m = l,$ we have $\mathbf{z}_m =
\mathbf{\hat z}_m$ and $\mathbf{u}_m $ is obtained by an update (or
several updates) of $\mathbf{u}_{m-1}+\mathbf{z}_m$, so that $$
J(\mathbf{u}_m) \le J(\mathbf{u}_{m-1}+\mathbf{z}_m) =
J(\mathbf{u}_{m-1}+\mathbf{\hat z}_m)
$$
Otherwise, for $\boldsymbol{\alpha}_m =r,$ we have $\mathbf{u}_m =
\mathbf{u}_{m-1} + \mathbf{z}_m$ with $\mathbf{z}_m$ obtained by an
update of $\mathbf{\hat z}_m$, such that
$$J(\mathbf{u}_m) = J(\mathbf{u}_{m-1}+\mathbf{z}_m) \le J(\mathbf{u}_{m-1}+\mathbf{\hat
z}_m)
$$
We begin with the following Lemma.

\begin{lemma}\label{lem:J_decrease}
Assume that $J$ satisfies (A1)-(A2). Then $\{J(\mathbf{u}_m)\}_{m\ge
1},$ where $\{\mathbf{u}_m\}_{m \ge 1}$ is a progressive
Proper Generalized Decomposition, over $\Sc_1$,  of
$$\mathbf{u} =
\arg \min_{\mathbf{v} \in \mathbf{V}_{\|\cdot\|}} J(\mathbf{v}),
$$
is
a non increasing sequence:
$$
J(\mathbf{u}_{m})\le J(\mathbf{u}_{m-1}) \text{ for all } m \ge 1.
$$
Moreover, if $J(\mathbf{u}_m)=J(\mathbf{u}_{m-1})$,
$\mathbf{u}_{m-1}$ is the solution of \eqref{eq:prob}.
\end{lemma}
\begin{proof}
By definition, we have
$$
J(\mathbf{u}_m)\le J(\mathbf{u}_{m-1}+\mathbf{z}_m) \le
J(\mathbf{u}_{m-1}+\mathbf{\hat z}_m)\le
J(\mathbf{u}_{m-1}+\mathbf{z}) \quad \forall \mathbf{z}\in\Sc_1
$$
In particular, since $\mathbf{0}\in\Sc_1$ by assumption (B1), we
have $J(\mathbf{u}_m)\le J(\mathbf{u}_{m-1})$. If $J(\mathbf{u}_m)=
J(\mathbf{u}_{m-1})$, we have
$$
J(\mathbf{u}_{m-1}) = \min_{\mathbf{z}\in \Sc_1}
J(\mathbf{u}_{m-1}+\mathbf{z})
$$
and by Lemma \ref{lem:carac_sol_S1}, we have that $\mathbf{u}_{m-1}$
solves \eqref{eq:prob}.
\qed \end{proof}

\begin{remark}
If $J(\mathbf{u}_m)=J(\mathbf{u}_{m-1})$ holds for some $m > 1,$
that is $\mathbf{u}_{m-1}$ is the solution of \eqref{eq:prob}, then
the updated PGD is described by a finite sequence of symbols
$\underline{\boldsymbol{\alpha}} (\mathbf{u}) =\boldsymbol{\alpha}_1
\boldsymbol{\alpha}_2 \cdots \boldsymbol{\alpha}_{m-1},$ where
$\boldsymbol{\alpha}_k \in \{c,l,r\}$ for $1 \le k \le m-1.$
Otherwise, $\{J(\mathbf{u}_m)\}_{m \in \mathbb{N}}$ is a strictly
decreasing sequence of real numbers and
$\underline{\boldsymbol{\alpha}}(\mathbf{u}) \in
\{c,l,r\}^{\mathbb{N}}.$
\end{remark}

\begin{definition}
Let $\boldsymbol{\alpha} \in \{c,l,r\}.$ Then
$\boldsymbol{\alpha}^{\infty} \in \{c,l,r\}^{\mathbb{N}}$ denotes
the infinite sequence of symbols $\boldsymbol{\alpha} \,
\boldsymbol{\alpha} \cdots \boldsymbol{\alpha} \cdots .$
\end{definition}

From now on, we will distinguish two convergence studies, one with a
weak continuity assumption on functional $J$, the other one without
weak continuity assumption on $J$ but with an additional Lipschitz
continuity assumption on the differential $J'$.

\subsubsection{A first approach for weakly sequentially
continuous  functionals}

Here, we introduce the following assumption.
\begin{itemize}
\item[(A3)] The map $J:\mathbf{V}_{\|\cdot\|} \longrightarrow \mathbb{R}$
is weakly sequentially continuous.
\end{itemize}

\begin{theorem}\label{th:convergence_weakly_continuous}
Assume that $J$ satisfies (A1)-(A3). Then every progressive Proper
Generalized Decomposition $\{\mathbf{u}_m\}_{m \ge 1},$ over $\Sc_1,$  of
$$\mathbf{u} = \arg \min_{\mathbf{v} \in \mathbf{V}_{\|\cdot\|}}
J(\mathbf{v})$$ converges in $\mathbf{V}_{\|\cdot\|}$ to
$\mathbf{u},$ that is,
$$
\lim_{m\rightarrow\infty}\Vert \mathbf{u} -\mathbf{u}_m\Vert
= 0
$$
\end{theorem}
\begin{proof}
From Lemma \ref{lem:J_decrease}, $\{J(\mathbf{u}_{m})\}$ is a non
increasing sequence. If there exists $m$ such that
$J(\mathbf{u}_{m})=J(\mathbf{u}_{m-1})$, from Lemma
\ref{lem:J_decrease}, we have $\mathbf{u}_m=\mathbf{u}$, which ends
the proof. Let us now suppose that
$J(\mathbf{u}_m)<J(\mathbf{u}_{m-1})$ for all $m$. $J(\mathbf{u}_m)$
is a strictly decreasing sequence which is bounded below by
$J(\mathbf{u})$. Then, there exists
$$
J^*=\lim_{m\to \infty} J(\mathbf{u}_m) \ge J(\mathbf{u}).
$$
If $J^*=J(\mathbf{u})$, Lemma \ref{lem:strong_convergence} allows to
conclude that $\mathbf{u}_m \rightarrow \mathbf{u}.$ Therefore, it
remains to prove that $J^*=J(\mathbf{u})$. Since $J$ is coercive,
the sequence $\{\mathbf{u}_m\}_{m\in\Nbb}$ is bounded in
$\mathbf{V}_{\|\cdot\|}$. Then, there exists a subsequence
$\{\mathbf{u}_{m_k}\}_{k\in\Nbb}$ that weakly converges to some
$\mathbf{u}^*\in V$. Since $J$ is weakly sequentially continuous, we
have
$$J^* = \lim_{k\to\infty} J(\mathbf{u}_{m_k}) = J(\mathbf{u}^*).$$ By definition of
the PGD, we have for all $\mathbf{z}\in \Sc_1$,
$$
J(\mathbf{u}_{m_{(k+1)}}) \le J(\mathbf{u}_{m_{k}+1}) \le
J(\mathbf{u}_{m_{k}}+\mathbf{z})
$$
Taking the limit with $k$, and using the weak sequential continuity of $J$, we
obtain $$ J(\mathbf{u}^*) \le J(\mathbf{u}^*+\mathbf{z}) \quad
\forall \mathbf{z}\in \Sc_1,$$ and by Lemma \ref{lem:carac_sol_S1},
we obtain $\mathbf{u}^*=\mathbf{u}$ and a fortiori
$J(\mathbf{u}^*)=J(\mathbf{u}),$ that concludes the proof.
\qed \end{proof}

\subsubsection{A second approach for
a class of functionals with Lipschitz continuous derivative on
bounded sets}

Now, assume that 
assumption (A3) is replaced by
\begin{itemize}
\item[(A3)] $J':\mathbf{V}_{\|\cdot\|} \longrightarrow \mathbf{V}_{\|\cdot\|}^\ast$ is Lipschitz continuous
on bounded sets, i.e. for $A$ a bounded set in
$\mathbf{V}_{\|\cdot\|}$, there exists a constant $C_A>0$ such that
\begin{align}
\Vert J'(\mathbf{v})-J'(\mathbf{w}) \Vert \;\le\; C_A \Vert
\mathbf{v}-\mathbf{w}\Vert\label{eq:J_lipschitz}
\end{align}
for all $\mathbf{v},\mathbf{w} \in A.$
\end{itemize}
The next three lemmas will give some useful properties of the
sequence $\{\mathbf{z}_m\}_{m \ge 1}.$

\begin{lemma}\label{lem:progressive_euler_ortho}
Assume that $J$ satisfies (A1)-(A2) and let $\{\mathbf{u}_m\}_{m \ge
1}$ be a progressive Proper Generalized Decomposition, over
$\Sc_1$,  of
$$\mathbf{u} = \arg \min_{\mathbf{v} \in \mathbf{V}_{\|\cdot\|}}
J(\mathbf{v}).$$ Then
$$
\langle J'(\mathbf{u}_{m-1}+\mathbf{z}_m), \mathbf{z}_m \rangle = 0,
$$
for all $m\ge 1.$
\end{lemma}
\begin{proof}
Let $\mathbf{z}_m= \lambda_m \mathbf{w}_m$, with
$\lambda_m\in\Rbb^+$ and $\norm[\mathbf{w}_m]=1$. In the cases
$\boldsymbol{\alpha}_m=c$ (purely progressive PGD) and
$\boldsymbol{\alpha}_m =l,$ we have $\mathbf{z}_m =\mathbf{\hat z}_m
\in \arg\min_{\mathbf{z}\in \Sc_1} J(\mathbf{u}_{m-1}+\mathbf{z})$.
From assumption (B2), we obtain
$$J(\mathbf{u}_{m-1}+\lambda_m \mathbf{w}_m)\le
J(\mathbf{u}_{m-1}+\lambda \mathbf{w}_m)$$ for all $\lambda \in
\Rbb$. This inequality is also true for $\boldsymbol{\alpha}_m = r$
since $\mathbf{z}_m = \arg\min_{\mathbf{z}\in\mathbf{U}(\mathbf{\hat
z}_m)} J(\mathbf{u}_{m-1}+\mathbf{z})$ and $\mathbf{U}(\mathbf{\hat
z}_m)$ is a linear space. Taking $\lambda = \lambda_m \pm \gamma$,
with $\gamma\in \Rbb^+$, we obtain for all cases
\begin{align*}
0&\le \frac{1}{\gamma} \left(J(\mathbf{u}_{m-1}+\lambda_m
\mathbf{w}_m\pm \gamma \mathbf{w}_m)- J(\mathbf{u}_{m-1}+\lambda_m
\mathbf{w}_m)\right).
\end{align*}
Taking the limit $\gamma \searrow 0$, we obtain $ 0 \le \pm \langle
J'(\mathbf{u}_{m-1}+\lambda_m \mathbf{w}_m),\mathbf{w}_m\rangle $ and
therefore
$$\langle J'(\mathbf{u}_{m-1}+\lambda_m \mathbf{w}_m),\mathbf{w}_m\rangle = 0,$$
which ends the proof.
\qed \end{proof}

\begin{lemma}\label{lem:progressive_lim_zm}
Assume that $J$ satisfies (A1)-(A2). Then the corrections
$\{\mathbf{z}_m\}_{m \ge 1}$ of a progressive Proper
Generalized Decomposition $\{\mathbf{u}_m\}_{m \ge 1},$ over
$\Sc_1,$  of
$$
\mathbf{u} = \arg \min_{\mathbf{v} \in
\mathbf{V}_{\|\cdot\|}} J(\mathbf{v}),
$$
satisfy
\begin{align} \sum_{m=1}^\infty \Vert  \mathbf{z}_m \Vert^s < \infty,
\text{ for some } s > 1,
\label{eq:summable_zm}
\end{align}
and thus,
\begin{align}\lim_{m\to\infty
} \Vert  \mathbf{z}_m\Vert = 0. \label{eq:lim_zm}
\end{align}
\end{lemma}

\begin{proof}
By the ellipticity property \eqref{eq:J_strong_convexity}, we have
\begin{align}
J(\mathbf{u}_{m-1})-J(\mathbf{u}_{m-1}+\mathbf{z}_m) &\ge \langle -
J'(\mathbf{u}_{m-1}+\mathbf{z}_m),\mathbf{z}_m\rangle +
\frac{\alpha}{s} \Vert \mathbf{z}_m \Vert^s \nonumber
\end{align}
for some $s > 1$ and $\alpha > 0.$
Using Lemma \ref{lem:progressive_euler_ortho} and
$J(\mathbf{u}_m)\le J(\mathbf{u}_{m-1}+\mathbf{z}_m)$, we then
obtain
\begin{align}
J(\mathbf{u}_{m-1})-J(\mathbf{u}_{m}) &\ge \frac{\alpha}{s} \Vert
 \mathbf{z}_m \Vert^s\label{eq:zms}
\end{align}
Now, summing on $m$, and using $\lim_{m\to\infty } J(\mathbf{u}_m) =
J^* <\infty $, we obtain
\begin{align*}
\frac{\alpha}{s} \sum_{m=1}^\infty  \Vert  \mathbf{z}_m \Vert^s &\le
\sum_{m=1}^\infty (J(\mathbf{u}_{m-1})-J(\mathbf{u}_{m}))  = J(0) -
J^* < +\infty.
\end{align*}
which implies $\lim_{m \rightarrow \infty}\Vert  \mathbf{z}_m
\Vert^s =0.$ The continuity of the map $x \mapsto x^{1/s}$ at $x=0$
proves (\ref{eq:lim_zm}).
\qed \end{proof}

\begin{lemma}\label{lem:convergence_maj}
Assume that $J$ satisfies (A1)-(A3). Then for every progressive
Proper Generalized Decompositions $\{\mathbf{u}_m\}_{m \ge 1},$ over
$\Sc_1,$  of $\mathbf{u} = \arg \min_{\mathbf{v} \in
\mathbf{V}_{\|\cdot\|}} J(\mathbf{v})$,  there exists $C >0$ such
that for $m\ge 1,$
$$ \vert \langle J'(\mathbf{u}_{m-1}),\mathbf{z}\rangle \vert \leqslant C \Vert  \mathbf{z}_m
\Vert \Vert \mathbf{z}\Vert,
$$
for all $\mathbf{z} \in \Sc_1$.
\end{lemma}
\begin{proof}
Since $J(\mathbf{u}_m)$ converges and since $J$ is coercive,
$\{\mathbf{u}_m\}_{m\ge 1}$ is a bounded sequence. Since $\Vert
\mathbf{z}_m \Vert \rightarrow 0$ as $m\to\infty$ (Lemma
\ref{lem:progressive_lim_zm}), $\{\mathbf{z}_m\}_{m\ge 1}$ is also a
bounded sequence. Let $a>0$ such that $\sup_{m} \norm[\mathbf{u}_m]+
\sup_{m}\norm[\mathbf{z}_m] \le a$ and let $C_B$ be the Lipschitz
continuity constant of $J'$ on the bounded set $B= \{\mathbf{v}\in
\mathbf{V}_{\|\cdot\|}:\norm[\mathbf{v}]\le a\}$. Then
\begin{align*}
-\langle J'(\mathbf{u}_{m-1}),\mathbf{z}\rangle  &= \langle
J'(\mathbf{u}_{m-1}+\mathbf{z})- J'(\mathbf{u}_{m-1})
 ,\mathbf{z}\rangle - \langle J'(\mathbf{u}_{m-1}+\mathbf{z}),\mathbf{z}\rangle
\\
& \le C_B \Vert \mathbf{z} \Vert^2 - \langle
J'(\mathbf{u}_{m-1}+\mathbf{z}),\mathbf{z}\rangle
\end{align*}
for all $\mathbf{z}\in  A = \{\mathbf{z}\in \Sc_1 :
\norm[\mathbf{z}]\le \sup_m \norm[\mathbf{z}_m]\}$. By convexity of
$J$ and since $J(\mathbf{u}_{m-1}+\mathbf{z}_m)\le
J(\mathbf{u}_{m-1}+\mathbf{z})$ for all $\mathbf{z}\in\Sc_1$, we
have
\begin{align*}
\langle J'(\mathbf{u}_{m-1}+\mathbf{z}),\mathbf{z}_m-\mathbf{z}
\rangle \le J(\mathbf{u}_{m-1}+\mathbf{z}_m) -
J(\mathbf{u}_{m-1}+\mathbf{z}) \le 0
\end{align*}
Therefore, for all $\mathbf{z}\in  A$, we have
\begin{align*}
-\langle J'(\mathbf{u}_{m-1}),\mathbf{z}\rangle  &\le
 C_B
\Vert \mathbf{z} \Vert^2 - \langle
J'(\mathbf{u}_{m-1}+\mathbf{z}),\mathbf{z}_m\rangle\\
&\le C_B \Vert \mathbf{z} \Vert^2 - \langle
J'(\mathbf{u}_{m-1}+\mathbf{z}) -
J'(\mathbf{u}_{m-1}+\mathbf{z}_m),\mathbf{z}_m\rangle
\quad \text{\footnotesize (Lemma \ref{lem:progressive_euler_ortho})} \\
&\le C_B \Vert \mathbf{z} \Vert^2 + C_B\Vert \mathbf{z}-\mathbf{z}_m
\Vert \norm[\mathbf{z}_m] \quad \text{\footnotesize (Lemma
\ref{lem:progressive_lim_zm})}
\\
&\le C_B \left(\norm[\mathbf{z}] ^2 +  \norm[\mathbf{z}]
\norm[\mathbf{z}_m] + \norm[\mathbf{z}_m]^2\right)
\end{align*}
Let $\mathbf{z}=\mathbf{w} \norm[\mathbf{z}_m] \in A$, with
$\norm[\mathbf{w}] = 1$. Then
\begin{align*}
\vert \langle J'(\mathbf{u}_{m-1}),\mathbf{w}\rangle \vert  & \le
3C_B \norm[\mathbf{z}_m] \quad \forall \mathbf{w}\in\{\mathbf{w}\in
\Sc_1:\norm[\mathbf{w}] = 1\}
\end{align*}
Taking $\mathbf{w}=\mathbf{z}/\norm[\mathbf{z}]$, with
$\mathbf{z}\in \Sc_1$, and $C=3C_B > 0$ we obtain
 \begin{align*}
\vert \langle J'(\mathbf{u}_{m-1}),\mathbf{z}\rangle \vert  &\le C
\norm[\mathbf{z}_m]\norm[\mathbf{z}] \quad \forall
\mathbf{z}\in\Sc_1
\end{align*}
\qed \end{proof}

Since $\mathbf{V}_{\|\cdot\|}$ is reflexive, we can identify
$\mathbf{V}_{\|\cdot\|}^{**}$ with $\mathbf{V}_{\|\cdot\|}$ and the
duality pairing $\langle \cdot, \cdot
\rangle_{\mathbf{V}_{\|\cdot\|}^{**},\mathbf{V}_{\|\cdot\|}^*}$ with
$\langle \cdot, \cdot
\rangle_{\mathbf{V}_{\|\cdot\|}^*,\mathbf{V}_{\|\cdot\|}}$
(\textit{i.e.} weak and weak-$*$ topologies coincide on
$\mathbf{V}_{\|\cdot\|}^*$).

\begin{lemma}\label{lem:convergence_Jprime}
Assume that $J$ satisfies (A1)-(A3). Then for every
progressive Proper Generalized Decomposition $\{\mathbf{u}_m\}_{m
\ge 1},$ over $\Sc_1,$  of $\mathbf{u} = \arg \min_{\mathbf{v} \in
\mathbf{V}_{\|\cdot\|}} J(\mathbf{v})$, the sequence
$\{J'(\mathbf{u}_m)\}_{m\in\Nbb}$ weakly-$*$ converges to
$\mathbf{0}$ in $\mathbf{V}_{\|\cdot\|}^*,$  that is, $
\lim_{m\rightarrow \infty}\langle J'(\mathbf{u}_{m}), \mathbf{z}
\rangle = 0 $ for all $\mathbf{z}$ in a dense subset of
$\mathbf{V}_{\|\cdot\|}.$
\end{lemma}
\begin{proof}
The sequence $\{\mathbf{u}_m\}_{m\in \Nbb}$ being bounded, and since
$J'$ is Lipschitz continuous on bounded sets, we have that there
exists a constant $C>0$ such that
$$
\Vert J'(\mathbf{u}_m)\Vert = \Vert J'(\mathbf{u}_m)
-J'(\mathbf{u})\Vert \le C \Vert \mathbf{u}-\mathbf{u}_m\Vert
$$
That proves that $\{J'(\mathbf{u}_m)\}\subset
\mathbf{V}_{\|\cdot\|}^*$ is a bounded sequence. Since
$\mathbf{V}_{\|\cdot\|}^*$ is also reflexive, from any subsequence
of $\{J'(\mathbf{u}_m)\}_{m\in \Nbb}$, we can extract a further
subsequence $\{J'(\mathbf{u}_{m_k})\}_{k\in \Nbb}$ that weakly-$*$
converges to an element $ \boldsymbol{\varphi} \in
\mathbf{V}_{\|\cdot\|}^*$. By using Lemma \ref{lem:convergence_maj},
we have for all $\mathbf{z}\in \Sc_1$,
$$
\vert \langle J'(\mathbf{u}_{m_k}),\mathbf{z}\rangle \vert  \leq
C \Vert \mathbf{z}_{m_k+1} \Vert \Vert \mathbf{z}\Vert.
$$
Taking the limit with $k$, and using Lemma
\ref{lem:progressive_lim_zm}, we obtain
$$
\langle \boldsymbol{\varphi},\mathbf{z}\rangle = 0\quad \forall\mathbf{z}\in \Sc_1,
$$
By using assumption (B3), we conclude that
$\boldsymbol{\varphi}=\mathbf{0}$. Since from any subsequence of the
initial sequence $\{J'(\mathbf{u}_m)\}_{m\in\Nbb}$ we can extract a
further subsequence that weakly-$*$ converges to the same limit
$\mathbf{0}$, then the whole sequence converges to $\mathbf{0}.$
\qed \end{proof}

\begin{lemma}\label{lem:convergence_Jprime_bis}
Assume that $J$ satisfies (A1)-(A3) and consider a progressive
Proper Generalized Decomposition $\{\mathbf{u}_m\}_{m \ge 1},$ over
$\Sc_1,$  of $\mathbf{u} = \arg \min_{\mathbf{v} \in
\mathbf{V}_{\|\cdot\|}} J(\mathbf{v})$ such that for the ellipticity
constant $s$ of $J$ and $\underline{\boldsymbol{\alpha}}(\mathbf{u})
= \boldsymbol{\alpha}_1 \cdots \boldsymbol{\alpha}_{m} \cdots,$ one
of the two following conditions hold:
\begin{itemize}
\item[(a)] $s>1$ and there exists a subsequence
$\{\boldsymbol{\alpha}_{m_k}\}_{k \in \mathbb{N}}$ such that
$\boldsymbol{\alpha}_{m_k} = l$ for all $k \ge 1.$
\item[(b)] $1< s \le 2 $ and there exists $k \ge 1$ such that
$\underline{\boldsymbol{\alpha}}(\mathbf{u})
=\boldsymbol{\alpha}_1 \cdots \boldsymbol{\alpha}_{k-1}
\boldsymbol{\alpha}^{\infty}$ where $\boldsymbol{\alpha} \in  \{c,r\}.$
\end{itemize}
Then, there exists a subsequence $\{\mathbf{u}_{m_k}\}_{k\in \Nbb}$
such that
$$
\langle J'(\mathbf{u}_{m_k}),\mathbf{u}_{m_k} \rangle \rightarrow 0.
$$
\end{lemma}

\begin{proof}
First, assume that condition (a) holds. Recall that if
$\boldsymbol{\alpha}_m = l$ for some $m \ge 1$, the $\mathbf{u}_{m}$
is obtained by the minimization of $J$ on the closed subspace
$\mathbf{U}(\mathbf{u}_{m-1} + \mathbf{z}_{m}) \subset
\mathbf{V}_{\|\cdot\|}$. The global minimum is attained and unique,
and it is characterized by
$\left<J'(\mathbf{u}_m),\mathbf{v}\right>=0$ for all $\mathbf{v}\in
\mathbf{U}(\mathbf{u}_{m-1} + \mathbf{z}_m)$. Thus, under condition
$(a)$, there exists a subsequence such that $\langle
J'(\mathbf{u}_{m_k}),\mathbf{u}_{m_k}\rangle=0$ for all $k \ge 1$.
Now, we consider that statement (b) holds. Without loss of
generality we may assume that
$\underline{\boldsymbol{\alpha}}(\mathbf{u}) =
\boldsymbol{\alpha}^{\infty}$ where $\boldsymbol{\alpha} \in
\{c,r\}.$ In both cases, $\mathbf{u}_m = \sum_{k=1}^m \mathbf{z}_k.$
Thus, we have
\begin{align*}
\vert\langle J'(\mathbf{u}_{m}),\mathbf{u}_{m} \rangle \vert & \le  \sum_{k=1}^m
\vert \langle J'(\mathbf{u}_{m}),\mathbf{z}_k \rangle \vert\\
&\le  C \sum_{k=1}^m \Vert \mathbf{z}_{m+1}\Vert \Vert \mathbf{z}_k \Vert \quad
\text{(By Lemma \ref{lem:convergence_maj})}.
\end{align*}
Let $s^*>1$ be such that $1/s^*+1/s =
1$. By Holder's inequality, we have
\begin{align}
\vert\langle J'(\mathbf{u}_{m}),\mathbf{u}_{m} \rangle \vert  &\le  C \Vert \mathbf{z}_{m+1}\Vert
m^{1/s^*}\left( \sum_{k=1}^m \Vert  \mathbf{z}_k \Vert^s\right)^{1/s} \nonumber
\\
&= C \left( m\Vert \mathbf{z}_{m+1}\Vert^{s^*} \right)^{1/s^*} \left(
\sum_{k=1}^m \Vert \mathbf{z}_k \Vert^s\right)^{1/s}. \label{ahora}
\end{align}
From Lemma \ref{lem:progressive_lim_zm}, we have $
\sum_{k=1}^{\infty} \Vert\mathbf{z}_k \Vert^s <\infty.$ Then there
exists a subsequence such that $m_k\Vert
\mathbf{z}_{m_k+1}\Vert^{s}\rightarrow 0.$ For $1<s\le 2$, we have
$s\le s^*$. Since $\lim_{k\to\infty} \Vert \mathbf{z}_ {k} \Vert
=0$, we have $\Vert \mathbf{z}_{k}\Vert^{s^*} \le \Vert
\mathbf{z}_{k}\Vert^{s}$ for $k$ large enough, and therefore we also
have $m_k\Vert \mathbf{z}_{m_k+1}\Vert^{s^*}\rightarrow 0$, which
from \eqref{ahora} ends the proof of the lemma.
\qed \end{proof}

\begin{theorem}\label{th:convergence}
Assume that $J$ satisfies (A1)-(A3) and consider a progressive
Proper Generalized Decomposition $\{\mathbf{u}_m\}_{m \ge 1},$ over
$\Sc_1,$  of $\mathbf{u} = \arg \min_{\mathbf{v} \in
\mathbf{V}_{\|\cdot\|}} J(\mathbf{v})$ such that the ellipticity
constant $s$ of $J$ and
$\underline{\boldsymbol{\alpha}}(\mathbf{u})$ satisfy one of the
following conditions:
\begin{itemize}
\item[(a)] $s>1$ and there exists a subsequence $\{\boldsymbol{\alpha}_{m_k}\}_{k \in \mathbb{N}}$
such that $\boldsymbol{\alpha}_{m_k} = l$ for all $k \ge 1.$
\item[(b)] $1< s \le 2 $ and there exists $k \ge 1$ such that
$\underline{\boldsymbol{\alpha}}(\mathbf{u}) =\boldsymbol{\alpha}_1
\cdots \boldsymbol{\alpha}_{k-1} \boldsymbol{\alpha}^{\infty}$ where
$\boldsymbol{\alpha} \in  \{c,r\}.$
\item[(c)]$s>1$ and $\underline{\boldsymbol{\alpha}}(\mathbf{u})$ is finite.
\end{itemize}
Then $\{\mathbf{u}_m\}_{m \ge 1},$ converges in
$\mathbf{V}_{\|\cdot\|}$ to $\mathbf{u},$ that is,
$$
\lim_{m\rightarrow\infty}\Vert \mathbf{u} -\mathbf{u}_m\Vert = 0.
$$
\end{theorem}
\begin{proof}
From Lemma \ref{lem:J_decrease}, $\{J(\mathbf{u}_{m})\}$ is a non
increasing sequence. If (c) holds, there exists $m$ such that
$J(\mathbf{u}_{m})=J(\mathbf{u}_{m-1})$ and from Lemma
\ref{lem:J_decrease}, we have $\mathbf{u}_m=\mathbf{u}$, which ends
the proof. Let us now suppose that
$J(\mathbf{u}_m)<J(\mathbf{u}_{m-1})$ for all $m$. $J(\mathbf{u}_m)$
is a strictly decreasing sequence which is bounded below by
$J(\mathbf{u})$. Then, there exists
$$
J^*=\lim_{m\to \infty} J(\mathbf{u}_m) \ge J(\mathbf{u}).
$$
If $J^*=J(\mathbf{u})$, Lemma \ref{lem:strong_convergence} allows
to conclude that $\{\mathbf{u}_m\}$ strongly converges to
$\mathbf{u}$. Therefore, it remains to prove that
$J^*=J(\mathbf{u})$. By the convexity of $J$, we have
\begin{align*}
J(\mathbf{u}_{m})-J(\mathbf{u}) \le \langle
J'(\mathbf{u}_{m}),\mathbf{u}_{m}-\mathbf{u}\rangle = \langle
J'(\mathbf{u}_{m}),\mathbf{u}_{m}\rangle -    \langle
J'(\mathbf{u}_{m}), \mathbf{u}\rangle
\end{align*}
By Lemmas \ref{lem:convergence_Jprime} and
\ref{lem:convergence_Jprime_bis}, we have that there exists a
subsequence $\{\mathbf{u}_{m_k}\}_{k\in\Nbb}$ such that $\langle
J'(\mathbf{u}_{m_k}),\mathbf{u}_{m_k}\rangle\rightarrow 0$ and
$\langle J'(\mathbf{u}_{m_k}), \mathbf{u}\rangle \rightarrow 0$, and
therefore
$$
J^* - J(\mathbf{u}) = \lim_{k\to\infty } J(\mathbf{u}_{m_k}) -
J(\mathbf{u}) \le 0
$$
Since we already had $J^*\ge J(\mathbf{u})$, this yields
$J^*=J(\mathbf{u})$, which ends the proof.
\qed \end{proof}

\section{Examples}\label{sec:examples}

\subsection{On the Singular Value
Decomposition in $L^p$ spaces for $p \ge 2$} A Banach space $V$ is
said to be smooth if for any linearly independent elements $x,y\in
V$, the function $\phi(t)=\|x-ty\|$ is differentiable. A Banach
space is said to be uniformly smooth if its modulus of smoothness
$$
\rho(\tau) = \sup_{\substack{x,y \in V\\
\|x\|=\|y\|=1}}\left\{\frac{\|x+\tau y\|+\|x-\tau y\|}{2}-1
\right\},\quad \tau >0,
$$
satisfies the condition
$$
\lim_{\tau \rightarrow 0}\frac{\rho(\tau)}{\tau}= 0.
$$
In uniformly smooth spaces, and only in such spaces,
the norm is uniformly Fr\'echet differentiable. It can be shown that the $L^p$-spaces  for
$1 < p < \infty$ are uniformly smooth (see Corollary 6.12 in \cite{TMorrison}).

Following section \ref{BochnerS}, we introduce the tensor product of
Lebesgue spaces
$$
L^p_\mu(I_1 \times I_2)= L^p_{\mu_1}(I_1)\otimes_{\Delta_p}
L^p_{\mu_2}(I_2) = L^p_{\mu_1}(I_1, L^p_{\mu_2}(I_2)),
$$
with $p\ge 2$, and $\mu=\mu_1 \otimes \mu_2$ a
finite product measure. Recall that
$$\|\mathbf{v} \|_{\Delta_p} = \left(\int_{I_1 \times I_2} \vert \mathbf{v}(x)\vert^p d\mu(x)\right)^{1/p}
$$
Let $\mathbf{u}$ be a given function in $L^p_\mu(I_1 \times I_2)$. We introduce the
functional $J:L^p_\mu(I_1 \times I_2)\rightarrow \Rbb$ defined by
$$
J(\mathbf{v}) = \frac{1}{p}\norm[\mathbf{v}-\mathbf{u}]_{\Delta_p}^p.
$$
Let $G:L^p_\mu(I_1 \times I_2)\rightarrow \Rbb$ be the functional
given by the $p$-norm $G(\mathbf{v})=\|\mathbf{v}\|_{\Delta_p}.$ It
is well known (see for example page 170 in \cite{RHolmes}) that $G$
is Fr\'echet differentiable, with
$$
G'(\mathbf{v}) = \mathbf{v}\, |\mathbf{v}|^{p-2}\,
\|\mathbf{v}\|_{\Delta_p}^{1-p} \in L^q_{\mu}(I_1 \times I_2),
$$
with $q$ such that $1/q+1/p=1$. We denote by $\mathcal{C}^k$ the set
of $k$-times Fr\'echet differentiable functionals from $L^p_\mu(I_1
\times I_2)$ to $\mathbb{R}.$ Then, if $p$ is an even integer, $G
\in \mathcal{C}^{\infty}.$  Otherwise, when $p$ is not an even
integer, the following statements hold (see \cite{Bonic} and 13.13
in \cite{Kriegl}):
\begin{enumerate}
\item[(a)] If $p$ is an integer, $G$ is $(p- 1)$-times
differentiable with Lipschitzian highest Fr\'echet derivative.
\item[(b)] Otherwise, $G$ is $[p]$-times Fr\'echet  differentiable with highest derivative
being H\"olderian of order $p -[p]$.
\item[(c)] $G$ has no higher H\"older Fr\'echet differentiability properties.
\end{enumerate}

As a consequence we obtain that $G \in \mathcal{C}^2$  for all $p
\ge 2,$ and the functional $J$ is also Fr\'echet differentiable with
Fr\'echet derivative given by $J'(\mathbf{v}) =
G(\mathbf{v}-\mathbf{u})^{p-1}G'(\mathbf{v}-\mathbf{u}),$ that is,
$$
 \langle J'(\mathbf{v}),\mathbf{w} \rangle = \int_{I_1\times I_2}  (\mathbf{v}-\mathbf{u})\,
 \vert \mathbf{v} -\mathbf{u} \vert^{p-2} \, \mathbf{w} \, d\mu.
$$
Thus, $J$ satisfies assumption (A1).
 It is well-known that if a functional $F:V \longrightarrow W,$
where $V$ and $W$ are Banach spaces, is Fr\'echet differentiable at
$v \in V$, then it is also locally Lipschitz continuous at $v \in
V.$ Thus, if $p \ge 2$, we have that $J' \in \mathcal{C}^{1},$ and
as a consequence $J'$ satisfies (A3).

Finally, in order to prove the convergence of the (updated)
progressive PGD for each $\mathbf{u}\in L^p_\mu(I_1 \times I_2)$
over $\Sc_1 = \mathcal{T}_{(r_1,r_2)}(L^p_{\mu_1}(I_1)\otimes_{a}
L^p_{\mu_2}(I_2)),$ where $(r_1,r_2) \in \mathbb{N}^2,$ we have to
verify that (A2) on $J$ is satisfied. Since there exists a constant
$\alpha_p>0$ such that for all $s,t\in\Rbb$,
$$
(\vert s\vert^{p-2}s - \vert t\vert^{p-2}t)(s-t) \ge \alpha_p\vert
s-t\vert^p
$$
(see for example (7.1) in \cite{CAN05}), then, for all
$\mathbf{v},\mathbf{w}\in  L^p_\mu(I_1 \times I_2)$,
$$
\langle J'(\mathbf{v})-J'(\mathbf{w}),\mathbf{v}-\mathbf{w}\rangle \ge \alpha_p
\Vert \mathbf{v}-\mathbf{w}\Vert^p,
$$
which proves the ellipticity property of $J,$ and assumption (A2)
holds.
\par From Theorem \ref{th:convergence}, we conclude that the (updated)
progressive Proper Generalized Decomposition converges
\begin{itemize}
\item for all $p\ge 2$ if conditions
$(a)$ or $(c)$ of Theorem \ref{th:convergence} hold.
\item for $p=2$ if condition (b) of Theorem \ref{th:convergence}
holds.
\end{itemize}

Let us detail the application of the progressive PGD over
$$\Sc_1 =
\mathcal{T}_{(r_1,r_2)}(L^p_{\mu_1}(I_1)\otimes_{a}
L^p_{\mu_2}(I_2)).$$ We claim that in dimension $d=2$, we can only
consider the case $r_1=r_2=r$. The claim follows from the fact that
(see \cite{FalcoHackbusch}) for each $\mathbf{v} \in
L^p_{\mu_1}(I_1)\otimes_{a} L^p_{\mu_2}(I_2),$ there exist two
minimal subspaces $U_{j,\min}(\mathbf{v}),$ $j=1,2$, with $\dim
U_{1,\min}(\mathbf{v}) = \dim U_{2,\min}(\mathbf{v})$ and such that
$\mathbf{v} \in U_{1,\min}(\mathbf{v}) \otimes_a
U_{2,\min}(\mathbf{v}).$ In consequence, for a fixed  $r \in
\mathbb{N}$ and for
$$
\mathbf{u} \in L^p_{\mu}(I_1 \times I_2) \setminus
L^p_{\mu_1}(I_1)\otimes_{a} L^p_{\mu_2}(I_2)
$$
we let
$$
\mathbf{u}_1 \in \arg \min_{\mathbf{z} \in
\mathcal{T}_{(r,r)}(L^p_{\mu_1}(I_1)\otimes_{a}
L^p_{\mu_2}(I_2))}J(\mathbf{z}).
$$
Then there exist two bases $\{u_1^{(j)},\ldots,u_r^{(j)}\} \subset
L^p_{\mu_j}(I_j)$ of $U_{j,\min}(\mathbf{v}),$ for $j=1,2,$ such
that
$$
\mathbf{u}_1 = \sum_{k=1}^{r} \sum_{l=1}^{r}\sigma_{k,l} \,
u_k^{(1)}\otimes u_l^{(2)},
$$
and $\mathbf{u}-\mathbf{u}_1 \notin L^p_{\mu_1}(I_1)\otimes_{a}
L^p_{\mu_2}(I_2).$
 Proceeding inductively we can write
$$
\mathbf{u}_m = \sum_{k=1}^{mr} \sum_{l=1}^{mr}\sigma_{k,l} \,
u_k^{(1)}\otimes u_l^{(2)}.
$$
At step $m$, an example of update of type $\boldsymbol{\alpha}_m =
r$ would consist in updating the coefficients $\{\sigma_{k,l}:k,l\in
\{(m-1)r+1,\ldots, mr\}\}$. An example of update of type
$\boldsymbol{\alpha}_m = l$ would consist in updating the whole set
of coefficients $\{\sigma_{k,l}:k,l\in\{1,\ldots,mr\}\}$.
\par In the case $p=2$ and when we take orthonormal bases, it
corresponds to the classical SVD decomposition in the Hilbert space
$L^2_\mu(I_1 \times I_2).$ In this case we have
$$
\mathbf{u}_m = \sum_{j=1}^{mr}\sigma_{j} \,  u_j^{(1)}\otimes
u_j^{(2)}.
$$
where
$\sigma_j = |\langle \mathbf{u},u_1^{(j)}\otimes u_2^{(j)}  \rangle|,$
for $1 \le j \le mr.$

In this sense, the progressive PGD can be interpreted as a SVD
decomposition of a function $\mathbf{u}$ in a $L^p$-space where $p
\ge 2$. Let us recall that for $p>2$, an update strategy of type
$(l)$ is required for applying Theorem \ref{th:convergence} (at
least for a subsequence of iterates).
\par
The above results can be
naturally extended to tensor product of Lebesgue spaces, $
\overline{{}_a\otimes_{k=1}^d
L^p_{\mu_k}(I_k)}^{\|\cdot\|_{\Delta_p}}$ with $d>2$
and $\Sc_1 = \mathcal{R}_{1}\left(\left._{a}%
\bigotimes_{k=1}^{d}L^p_{\mu_k}(I_{k})\right.\right),$ leading
to a generalization of multidimensional singular value
decomposition introduced in \cite{FAL10b} for the case of Hilbert
tensor spaces.

\subsection{Nonlinear Laplacian}

We here present an example taken from \cite{CAN05}. We refer to
section \ref{sec:sobolev} for the introduction to the properties of
Sobolev spaces. Let $\Omega = \Omega_1\times \hdots \times
\Omega_d$. Given some $p > 2$, we let
$\mathbf{V}_{\|\cdot\|} = H^{1,p}_0(\Omega)$, which
is the closure of $C_c^\infty(\Omega)$ (functions in
$C^\infty(\Omega)$ with compact support in $\Omega$) in
$H^{1,p}(\Omega)$ with respect to the norm in $H^{1,p}(\Omega)$. We
equip $H^{1,p}_0(\Omega)$ with  the norm
$$
\norm[\mathbf{v}] = \left(\sum_{k=1}^d
\norm[\partial_{x_k}(\mathbf{v})]_{L^p(\Omega)} \right)^{1/p}
$$
which is equivalent to the norm $\norm_{1,p}$ on $H^{1,p}(\Omega)$
introduced in section \ref{sec:sobolev}. We then introduce the
functional $J:\mathbf{V}_{\|\cdot\|} \rightarrow \Rbb$ defined by
$$
J(\mathbf{v}) = \frac{1}{p}\norm[\mathbf{v}]^p
- \langle \mathbf{f},\mathbf{v} \rangle,
$$
with $\mathbf{f}\in \mathbf{V}_{\|\cdot\|}^\ast$. Its Fr\'echet differential is
$$
J'(\mathbf{v}) = A(\mathbf{v}) -\mathbf{f}
$$
where
$$
A(\mathbf{v}) = -\sum_{k=1}^d  \frac{\partial}{\partial
x_k}\left(\left\vert \frac{\partial \mathbf{v}}{\partial
x_k}\right\vert^{p-2}  \frac{\partial \mathbf{v}}{\partial
x_k}\right)
$$
$A$ is called the $p$-Laplacian. Assumptions (A1)-(A3) on the
functional are satisfied (see \cite{CAN05}). Assumption (B3) on the
set $\Rc_1\left( \left._a
\bigotimes_{j=1}^{d}H^{m,p}_0(\Omega_{j})\right. \right)$ is also
satisfied. Indeed, it can be easily proved from
Proposition~\ref{prop:S1W1p_closed} that the set $\Rc_1\left(
\left._a \bigotimes_{j=1}^{d}H^{1,p}_0(\Omega_{j})\right. \right)$
is weakly closed in $(H^{1,p}_0(\Omega),\norm_{1,p})$. Since the
norm $\norm$ is equivalent to $\norm_{1,p}$ on $H^{1,p}_0(\Omega)$,
it is also weakly closed in $(H^{1,p}_0(\Omega),\norm)$.
\par
Then, from Theorem \ref{th:convergence}, the progressive PGD
converges if there exists a subsequence of updates of type $(l)$.

\subsection{Linear elliptic variational problems on Hilbert spaces}
Let $\mathbf{V}_{\|\cdot\|} = \overline{V_1\otimes_a \hdots \otimes_a V_d}^{\|\cdot\|}$ be a tensor product of Hilbert spaces. We consider the following problem
$$
J(\mathbf{u}) = \min_{\mathbf{v}\in K} J(\mathbf{v}),\quad J(\mathbf{v}) =
\frac{1}{2}a(\mathbf{v},\mathbf{v}) - \ell(\mathbf{v})
$$
where $K\subset \mathbf{V}_{\|\cdot\|}$,
$a:\mathbf{V}_{\|\cdot\|}\times \mathbf{V}_{\|\cdot\|}\rightarrow \Rbb$ is
a coercive
continuous symmetric bilinear form,
\begin{align*}
&a(\mathbf{v},\mathbf{v}) \ge \alpha \Vert \mathbf{v} \Vert^2
\quad \forall \mathbf{v}\in \mathbf{V}_{\|\cdot\|} ,\\
&a(\mathbf{v},\mathbf{w}) \le \beta \Vert \mathbf{v} \Vert\Vert \mathbf{w}
\Vert \quad \forall \mathbf{v},\mathbf{w}\in \mathbf{V}_{\|\cdot\|},
\end{align*}
$\ell:\mathbf{V}_{\|\cdot\|}\rightarrow \Rbb$ is a continuous linear form,
\begin{align*}
&\ell(\mathbf{v})\le \gamma \Vert \mathbf{v}\Vert \quad \forall \mathbf{v}
\in \mathbf{V}_{\|\cdot\|}.
\end{align*}

\paragraph{Case where $K$ is a closed and convex subset of $\mathbf{V}_{\|\cdot\|}$.}
The solution $\mathbf{u}$ is equivalently characterized by the variational
inequality
$$
a(\mathbf{u},\mathbf{v}-\mathbf{u}) \ge \ell(\mathbf{v}-\mathbf{u})
\quad \forall \mathbf{v}\in K
$$
In order to apply the results of the present paper, we have to
recast the problem as an optimization problem in $\mathbf{V}_{\|\cdot\|}$.
We introduce a
convex and Fr\'echet differentiable functional $j:\mathbf{V}_{\|\cdot\|}
\rightarrow \Rbb$
with Fr\'echet differential $j':\mathbf{V}_{\|\cdot\|}\rightarrow
\mathbf{V}_{\|\cdot\|}^\ast$, such that
$j(\mathbf{v})=0$ if $\mathbf{v}\in K$ and $j(\mathbf{v})>0$ if
$\mathbf{v}\notin K$. We further assume
that $j'$ is Lipschitz on bounded sets. We let $j_\epsilon(\mathbf{v}) =
\epsilon^{-1}j(\mathbf{v})$, with $\epsilon>0$, and introduce the following
penalized problem
$$
J_\epsilon(\mathbf{u}_\epsilon) = \min_{\mathbf{v}\in \mathbf{V}_{\|\cdot\|}}
J_\epsilon(\mathbf{v}),\quad
J_\epsilon(\mathbf{v}) = J(\mathbf{v}) +  j_\epsilon(\mathbf{v})
$$
As $\epsilon\to 0$, $j_\epsilon$ tends to the indicator function of
set $K$ and $\mathbf{u}_\epsilon \to \mathbf{u}$
(see e.g. \cite{GLO84}). Assumptions
(A1)-(A2) are verified since $J_\epsilon$ is Fr\'echet
differentiable with Fr\'echet differential $J_\epsilon': \mathbf{V}_{\|\cdot\|}
\rightarrow
\mathbf{V}^\ast_{\|\cdot\|}$ defined by
$$
\langle J_\epsilon'(\mathbf{v}),\mathbf{z}\rangle = a(\mathbf{v},\mathbf{z})
-\ell(\mathbf{z}) + \langle
j_\epsilon'(\mathbf{v}),\mathbf{z}\rangle,
$$
and $J_\epsilon$ is elliptic since
$$
\langle J_\epsilon'(\mathbf{v}) -J_\epsilon'(\mathbf{w}),\mathbf{v}-\mathbf{w}
\rangle = a(\mathbf{v}-\mathbf{w},\mathbf{v}-\mathbf{w}) +
\langle j_\epsilon'(\mathbf{v})-j_\epsilon'(\mathbf{w}),\mathbf{v}-
\mathbf{w}\rangle \ge \alpha
\norm[\mathbf{v}-\mathbf{w}]^2
$$
Assumption (A3) comes from the continuity of $a$ and $\ell$ and from
the properties of $j'$.

\paragraph{Case where $K=\mathbf{V}_{\|\cdot\|}$.}
If $K=\mathbf{V}_{\|\cdot\|}$, we recover the classical case of linear elliptic
variational problems on Hilbert spaces analyzed in \cite{FAL10b}. In
this case, the bilinear form $a$ defines a norm
$\norm[\mathbf{v}]_a=\sqrt{a(\mathbf{v},\mathbf{v})}$ on $\mathbf{V}_{\|\cdot\|}$, equivalent to the norm $\norm$.
The functional $J$ is here equal to
$$
J(\mathbf{v}) = \frac{1}{2}\norm[\mathbf{u}-\mathbf{v}]_a^2 -
\frac{1}{2}\norm[\mathbf{u}]_a^2
$$ The progressive PGD can be interpreted as a generalized
Eckart-Young decomposition (generalized singular value
decomposition) with respect to this non usual metric, and defined
progressively by
$$
\norm[ \mathbf{u} - \mathbf{u}_{m} ]^2_a  = \min_{\mathbf{z}\in \Sc_1}
\norm[\mathbf{u}-\mathbf{u}_{m-1}-\mathbf{z}]^2_a
$$
We have
$$
J(\mathbf{u}_{m-1})-J(\mathbf{u}_m) = \frac{1}{2}\Vert \mathbf{z}_m\Vert_a^2 :=
\frac{1}{2}\sigma_m^2
$$
and
$$
\Vert \mathbf{u}-\mathbf{u}_m\Vert^2_a =  \Vert \mathbf{u}\Vert_a^2
- \sum_{k=1}^m \sigma_k^2 \underset{m\to \infty}{\longrightarrow} 0
$$
where $\sigma_m$ can be interpreted as the dominant singular value
of $(\mathbf{u}-\mathbf{u}_{m-1})\in \mathbf{V}_{\|\cdot\|}$. The
PGD method has been successfully applied to this class of problems
in different contexts: separation of spatial coordinates for the
solution of Poisson equation in high dimension \cite{AMM07,LEB09},
separation of physical variables and random parameters for the
solution of parameterized stochastic partial differential equations
\cite{NOU07}.

\section{Conclusion}
\label{sec:conclu} In this paper, we have considered the solution of
a class of convex optimization problems in tensor Banach spaces with
a family of methods called progressive Proper Generalized
Decomposition (PGD) that consist in constructing a sequence of
approximations by successively correcting approximations with
optimal elements in a given subset of tensors. We have proved the
convergence of a large class of PGD algorithms (including update
strategies) under quite general assumptions on the convex functional
and on the subset of tensors considered in the successive
approximations. The resulting succession of approximations has been
interpreted as a generalization of a multidimensional singular value
decomposition (SVD). Some possible applications have been
considered. \par Further theoretical investigations are still
necessary for a better understanding of the different variants of
PGD methods and the introduction of more efficient algorithms for
their construction (e.g. alternated direction algorithms, ...). The
analysis of algorithms for the solution of successive approximation
problems on tensor subsets is still an open problem. In the case of
dimension $d=2$, further analyses would be required in order to
better characterize the PGD as a direct extension of SVD when
considering more general norms.

\bibliographystyle{plain}

\begin{thebibliography}{10}

\bibitem{AMM10}
A.~Ammar, F.~Chinesta, and A.~Falc\'o.
\newblock On the convergence of a greedy rank-one update algorithm for a class
  of linear systems.
\newblock {\em Archives of Computational Methods in Engineering}, 17(4):473--486,
2010.

\bibitem{AMM07}
A.~Ammar, B.~Mokdad, F.~Chinesta, and R.~Keunings.
\newblock A new family of solvers for some classes of multidimensional partial
  differential equations encountered in kinetic theory modelling of complex
  fluids.
\newblock {\em Journal of Non-Newtonian Fluid Mechanics}, 139(3):153--176,
  2006.


\bibitem{BAL10}
J.~Ballani and L.~Grasedyck.
\newblock A projection method to solve linear systems in tensor format.
\newblock Technical Report 309, RWTH Aachen, 2010.

\bibitem{Blanchard2003}
P.~Blanchard and E.~Br\"uning.
\newblock Mathematical Methods in Physics: Distributions, Hilbert Space
Operators, and Variational Methods.
\newblock Birkh\"auser, 2003.

\bibitem{Bonic}
R.~Bonic and J.~Frampton.
\newblock Differentiable functions on certain Banach spaces,
\newblock {\em Bull. Am. Math. Soc.}, 71(2):393--395, 1965.

\bibitem{LEB09}
C.~Le Bris, T.~Lelievre, and Y.~Maday.
\newblock Results and questions on a nonlinear approximation approach for
  solving high-dimensional partial differential equations.
\newblock {\em Constructive Approximation}, 30(3):621--651, 2009.

\bibitem{CAN10}
E.~Cances, V.~Ehrlacher, and T.~Lelievre.
\newblock Convergence of a greedy algorithm for high-dimensional convex
  nonlinear problems.
\newblock {\em arXiv}, arXiv(1004.0095v1 [math.FA]), April 2010.

\bibitem{CAN05}
C.~Canuto and K.~Urban.
\newblock Adaptive optimization of convex functionals in banach spaces.
\newblock {\em SIAM J. Numer. Anal.}, 42(5):2043-2075, 2005.

\bibitem{CHI10}
F.~Chinesta, A.~Ammar, and E.~Cueto.
\newblock Recent advances in the use of the {Proper Generalized Decomposition}
  for solving multidimensional models.
\newblock {\em Archives of Computational Methods in Engineering}, 17(4):373--391,
  2010.

\bibitem{DefantFloret}
A.~Defant, and K.~Floret.
\newblock Tensor Norms and Operator ideals.
\newblock North-Holland, 1993.

\bibitem{DEL00}
L.~{De Lathauwer}, B.~{De Moor}, and J.~Vandewalle.
\newblock A multilinear singular value decomposition.
\newblock {\em SIAM J. Matrix Anal. Appl.}, 21(4):1253--1278, 2000.

\bibitem{DES08}
V.~de~Silva and L.-H. Lim.
\newblock Tensor rank and ill-posedness of the best low-rank approximation
  problem.
\newblock {\em SIAM Journal of Matrix Analysis \& Appl.}, 30(3):1084--1127,
  2008.

\bibitem{EKE99}
I.~Ekeland and R.~Teman.
\newblock {\em Convex Analysis and Variational Problems}.
\newblock Classics in Applied Mathematics. SIAM, 1999.

\bibitem{FalcoHackbusch}
A.~Falc\'o and W.~Hackbusch.
\newblock On minimal subspaces in tensor representations
\newblock {\em Preprint 70/2010} Max Planck Institute for Mathematics in
the Sciences, 2010.


\bibitem{FAL10b}
A.~Falc\'o and A.~Nouy.
\newblock A {Proper Generalized Decomposition} for the solution of elliptic
  problems in abstract form by using a functional {Eckart-Young} approach.
\newblock {\em J. Math. Anal. Appl.}, 376: 469--480,
2011.

\bibitem{GLO84}
R.~Glowinski.
\newblock {\em Numerical Methods for Nonlinear Variational Problems}.
\newblock Springer-Verlag, New York, 1984.

\bibitem{GRA10}
L. Grasedyck  \newblock Hierarchical Singular Value Decomposition of
Tensors
\newblock {\em SIAM J. Matrix Anal. Appl.}, 31:2029-2054, 2010.

\bibitem{HAC08b}
W. Hackbusch, .~N. Khoromskij, S.~A. Sauter, and E.~E.
  Tyrtyshnikov.
\newblock Use of tensor formats in elliptic eigenvalue problems.
\newblock Technical Report Research Report No.2010-78, Max Planck Institute for
  Mathematics in the Sciences, 2008.

\bibitem{RHolmes}
R.~Holmes.
\newblock {\em Geometric Functional Analysis and its Applications}.
\newblock Springer-Verlag, New York, 1975.


\bibitem{KHO10}
B.N. Khoromskij and C.~Schwab.
\newblock Tensor-structured galerkin approximation of parametric and stochastic
  elliptic pdes.
\newblock Technical Report Research Report No. 2010-04, ETH, 2010.


\bibitem{KOL09}
T.~G. Kolda and B.~W. Bader.
\newblock Tensor decompositions and applications.
\newblock {\em SIAM Review}, 51(3):455--500, 2009.

\bibitem{KOL01}
T.G. Kolda.
\newblock Orthogonal tensor decompositions.
\newblock {\em SIAM J. Matrix Analysis \& Applications}, 23(1):243--255, 2001.

\bibitem{Kriegl}
A.~Kriegl and P.W.~Michor.
\newblock{\em The Convenient Setting of Global Analysis}
\newblock Mathematical Surveys and Monographs Volume 53,
American Mathematical Society, 1997

\bibitem{LAD99b}
P.~Ladev\`eze.
\newblock {\em Nonlinear Computational Structural Mechanics - New Approaches
  and Non-Incremental Methods of Calculation}.
\newblock Springer Verlag, 1999.


\bibitem{LAD10}
P. Ladev\`eze, J.-C. Passieux and D. N�ron. \newblock  The LATIN
multiscale computational method and the Proper Generalized
Decomposition \newblock  {\em Computer Methods in Applied Mechanics
and Engineering}, 199:1287-1296, 2010.


\bibitem{MAT11}
H.~G. Matthies and E. Zander.
\newblock Sparse representations in stochastic mechanics.
\newblock In Manolis Papadrakakis, George Stefanou, and Vissarion Papadopoulos,
  editors, {\em Computational Methods in Stochastic Dynamics}, volume~22 of
  {\em Computational Methods in Applied Sciences}, pages 247--265. Springer
  Netherlands, 2011.

\bibitem{TMorrison}
T.~Morrison
\newblock {\em Functional Analysis. An Introduction to Banach Space Theory.}
\newblock John Wiley and Sons, 2001

\bibitem{NOU07}
A.~Nouy.
\newblock A generalized spectral decomposition technique to solve a class of
  linear stochastic partial differential equations.
\newblock {\em Computer Methods in Applied Mechanics and Engineering},
  196(45-48):4521--4537, 2007.

\bibitem{NOU08b}
A.~Nouy.
\newblock Generalized spectral decomposition method for solving stochastic
  finite element equations: invariant subspace problem and dedicated
  algorithms.
\newblock {\em Computer Methods in Applied Mechanics and Engineering},
  197:4718--4736, 2008.

\bibitem{NOU09d}
A.~Nouy. \newblock  Recent developments in spectral stochastic
methods for the numerical solution of stochastic partial
differential equations \newblock {\em Archives of Computational
Methods in Engineering}, 16:251-285, 2009.


\bibitem{NOU10}
A.~Nouy.
\newblock {Proper Generalized Decompositions} and separated representations for
  the numerical solution of high dimensional stochastic problems.
\newblock {\em Archives of Computational Methods in Engineering}, 17:403--434,
2010.


\bibitem{NOU10b}
A.~Nouy.
\newblock A priori model reduction through
Proper Generalized Decomposition for solving time-dependent partial
differential equations.
\newblock {\em Computer Methods in Applied Mechanics and
Engineering}, 199:1603-1626, 2010.

\bibitem{OSE10}
I. Oseledets and E. Tyrtyshnikov. \newblock
TT-cross approximation
for multidimensional arrays
\newblock {\em Linear Algebra And Its Applications},
 {2010}, {432}, {70-88}

\bibitem{RYA02}
R.~A. Ryan.
\newblock {\em Intoduction to tensor products of {Banach} spaces}.
\newblock Springer, 2002.

\bibitem{TEM08}
V. Temlyakov
\newblock Greedy Approximation
\newblock {\em Acta Numerica},17:235-409, 2008.



\bibitem{USC10}
Andre Uschmajew.
\newblock Well-posedness of convex maximization problems on stiefel manifolds
  and orthogonal tensor product approximations.
\newblock {\em Numerische Mathematik}, 115:309--331 2010.

\bibitem{Zeidler}
Eberhard Zeidler.
\newblock Nonlinear Functional Analysis and its Applications III. Variational
Methods and Optimization.
\newblock Springer-Verlag, 1985.
\end{thebibliography}

\end{document}